\documentclass[11pt]{article}
\usepackage{latexsym,amsfonts,amssymb,amsmath,amsthm}
\usepackage{graphicx}

\usepackage[usenames,dvipsnames]{color}
\usepackage{ulem}

\usepackage{color}
\begin{document}

\newtheorem{tm}{Theorem}[section]
\newtheorem{prop}[tm]{Proposition}
\newtheorem{defin}[tm]{Definition} % definition numbers are dependent on theorem numbers
\newtheorem{coro}[tm]{Corollary}
\newtheorem{lem}[tm]{Lemma}
\newtheorem{assumption}[tm]{Assumption}
\newtheorem{rk}[tm]{Remark}
\newtheorem{nota}[tm]{Notation}
\numberwithin{equation}{section}

\newcommand{\stk}[2]{\stackrel{#1}{#2}}
\newcommand{\dwn}[1]{{\scriptstyle #1}\downarrow}
\newcommand{\upa}[1]{{\scriptstyle #1}\uparrow}
\newcommand{\nea}[1]{{\scriptstyle #1}\nearrow}
\newcommand{\sea}[1]{\searrow {\scriptstyle #1}}
\newcommand{\csti}[3]{(#1+1) (#2)^{1/ (#1+1)} (#1)^{- #1
 / (#1+1)} (#3)^{ #1 / (#1 +1)}}
\newcommand{\RR}[1]{\mathbb{#1}}

\newcommand{\rd}{{\mathbb R^d}}
\newcommand{\ep}{\varepsilon}
\newcommand{\rr}{{\mathbb R}}
\newcommand{\alert}[1]{\fbox{#1}}
\newcommand{\eqd}{\sim}
\def\p{\partial}
\def\R{{\mathbb R}}
\def\N{{\mathbb N}}
\def\Q{{\mathbb Q}}
\def\C{{\mathbb C}}
\def\l{{\langle}}
\def\r{\rangle}
\def\t{\tau}
\def\k{\kappa}
\def\a{\alpha}
\def\la{\lambda}
\def\De{\Delta}
\def\de{\delta}
\def\ga{\gamma}
\def\Ga{\Gamma}
\def\ep{\varepsilon}
\def\eps{\varepsilon}
\def\si{\sigma}
\def\Re {{\rm Re}\,}
\def\Im {{\rm Im}\,}
\def\E{{\mathbb E}}
\def\P{{\mathbb P}}
\def\Z{{\mathbb Z}}
\def\D{{\mathbb D}}
\newcommand{\ceil}[1]{\lceil{#1}\rceil}

%\allowdisplaybreaks
\title{Global classical solutions, stability  of constant equilibria, and spreading speeds in attraction-repulsion chemotaxis systems with  logistic source on $\mathbb{R}^{N}$}

\author{
Rachidi B. Salako and Wenxian Shen  \\
Department of Mathematics and Statistics\\
Auburn University\\
Auburn University, AL 36849\\
U.S.A. \\
\\
Dedicated to the memory of Professor George Sell
}

\date{}
\maketitle
%\begin{document}

\begin{abstract}
In this paper, we consider the following chemotaxis systems of  parabolic-elliptic-elliptic type on $\mathbb{R}^{N}$,
$$
\begin{cases}
u_{t}=\Delta u -\chi_1\nabla( u\nabla v_1)+\chi_2 \nabla(u\nabla v_2 )+ u(a -b u), \qquad \  x\in\R^N,\ t>0,  \\
0=(\Delta- \lambda_1 I)v_1+ \mu_1 u,  \qquad \ x\in\R^N,\ t>0,  \\
0=(\Delta- \lambda_2 I)v_2+ \mu_2 u,  \qquad  \ x\in\R^N,\ t>0,  \\
u(\cdot,0)=u_{0}, \qquad x\in  \mathbb{R}^N ,
\end{cases}
$$
where $  \chi_{i}\geq 0,\ \lambda_i> 0,\ \mu_i>0$ ($i=1,2$) and  $\ a> 0,\  b> 0$ are constant real numbers, and $N$ is a positive integer. First, under some conditions on the parameters $\chi_i,\mu_i,\lambda_i, a, b$ and  $N$, we prove the global existence and boundedness of classical solutions $(u(x,t;u_0),v_1(x,t;u_0),v_2(x,t;u_0))$ for nonnegative, bounded, and uniformly continuous initial functions $u_0(x)$. Next, we explore the asymptotic stability of the constant equilibrium $(\frac{a}{b},\frac{\mu_1}{\lambda_1}\frac{a}{b},\frac{\mu_2}{\lambda_2}\frac{a}{b})$ and prove under some further assumption on the parameters that,  for every strictly  positive initial $u_0(x)$,
$$
\lim_{t\to\infty}\left[\|u(\cdot,t;u_0)-\frac{a}{b}\|_{\infty}+\|\lambda_{1}v_{1}(\cdot,t;u_0)-\frac{a}{b}\mu_1\|_{\infty}+\|\lambda_{2}v_{2}(\cdot,t;u_0)-\frac{a}{b}\mu_2\|_{\infty}\right]=0.
$$
Finally, we investigate the spreading properties of the global solutions with  compactly supported initial functions. We show that under some conditions on the parameters, there are two positive numbers $0<c^*_-(\chi_1,\mu_1,\lambda_1,\chi_2,\mu_2,\lambda_2)\le c^*_+(\chi_1,\mu_1,\lambda_1,\chi_2,\mu_2,\lambda_2)$ such that for every nonnegative initial function $u_0(x)$ with nonempty and compact support, we have
$$
\lim_{t\to\infty}\left[\sup_{|x|\leq ct}|u(x,t;u_0)-\frac{a}{b}| + \sup_{|x|\leq ct}|\lambda_1 v_1(x,t;u_0)-\frac{a}{b}\mu_1| + \sup_{|x|\leq ct}|\lambda_2 v_2(x,t;u_0)-\frac{a}{b}\mu_2|\right]=0
$$
 whenever $0\leq c< c^*_-(\chi_1,\mu_1,\lambda_1,\chi_2,\mu_2,\lambda_2)$,  and
$$
\lim_{t\to\infty}\left[\sup_{|x|\geq ct}|u(x,t;u_0)| + \sup_{|x|\geq ct}| v_1(x,t;u_0)| + \sup_{|x|\geq ct}|v_2(x,t;u_0)|\right]=0
$$
whenever $c>c^*_+(\chi_1,\mu_1,\lambda_1,\chi_2,\mu_2,\lambda_2)$. Furthermore we show that $$\lim_{(\chi_1,\chi_2)\to (0,0)}c^*_-(\chi_1,\mu_1,\lambda_1,\chi_2,\mu_2,\lambda_2)=\lim_{(\chi_1,\chi_2)\to (0,0)}c^*_+(\chi_1,\mu_1,\lambda_1,\chi_2,\mu_2,\lambda_2)=2\sqrt{a}. $$
\end{abstract}



\section{Introduction and the Statement of the Main Results}
Chemotaxis describes the oriented movement of biological cells or organism in response to chemical gradients. The oriented movement of cells has a crucial role in a wide range of biological phenomena. At the beginning of 1970s, Keller and Segel  (see \cite{KeSe1}, \cite{KeSe2}) introduced systems of
partial differential equations of the following form  to model the time evolution of both the density $u(x,t)$ of a mobile species and the density $v(x,t)$ of a  chemoattractant,
\begin{equation}\label{IntroEq0}
\begin{cases}
u_{t}=\nabla\cdot (m(u)\nabla u- \chi(u,v)\nabla v) + f(u,v),\quad   x\in\Omega, \cr
\tau v_t=\Delta v + g(u,v),\quad  x\in\Omega,
\end{cases}
\end{equation}
complemented with certain boundary condition on $\partial\Omega$ if $\Omega$ is bounded, where $\Omega\subset \R^N$ is an open domain;  $\tau\ge 0$ is a non-negative constant linked to the speed of diffusion of the chemical;  the function $\chi(u,v)$ represents  the sensitivity with respect to chemotaxis; and the functions $f$ and $g$ model the growth of the mobile species and the chemoattractant, respectively.
In literature, \eqref{IntroEq0} is called the Keller-Segel (KS) model or a chemotaxis model.

Since the works by Keller and Segel,  a rich variety of mathematical models for studying chemotaxis has appeared (see \cite{BBTW}, \cite{DiNa}, \cite{DiNaRa}, \cite{GaSaTe},  \cite{HiPa1}, \cite{HiPo}, \cite{KKAS},  \cite{NAGAI_SENBA_YOSHIDA}, \cite{Sug},  \cite{SuKu},  \cite{TeWi},   \cite{WaMuZh},  \cite{win_jde}, \cite{win_JMAA_veryweak}, \cite{win_arxiv}, \cite{Win}, \cite{win_JNLS}, \cite{YoYo}, \cite{ZhMuHuTi}, and the references therein).
  The reader is referred to \cite{HiPa, Hor} for some detailed introduction into the mathematics of KS models. In the current paper,  we consider chemoattraction-repulsion process in which cells undergo random motion and chemotaxis towards attractant and away from repellent \cite{MLACLE}. Moreover, we consider the model with proliferation and death of cells and assume that chemicals diffuse very quickly. These lead to the model of partial differential equations as follows:
\begin{equation}\label{Main Intro-eq0}
\begin{cases}
u_{t}=\Delta u -\chi_1\nabla( u\nabla v_1)+\chi_2 \nabla(u\nabla v_2 )+ u(a -b u), \qquad \  x\in\Omega,\ t>0,  \\
0=(\Delta- \lambda_1 I)v_1+ \mu_1 u ,  \qquad \ x\in\Omega,\ t>0,  \\
0=(\Delta- \lambda_2 I)v_2+ \mu_2 u , \qquad \text{in}\ \ x\in\Omega,\ t>0 ,
\end{cases}
\end{equation}
complemented with certain boundary condition on $\partial\Omega$ if $\Omega$ is bounded.

When $\Omega$ is a smooth bounded domain, it is  seen that  \eqref{Main Intro-eq0}  complemented with Neumann boundary conditions
\begin{equation}\label{Main Intro-eq1}
\frac{\partial u}{\partial n}=\frac{\partial v_1}{\partial n}=\frac{\partial v_2}{\partial n}=0,
\end{equation}
has a unique nonzero constant equilibrium  $(\frac{a}{b},\frac{\mu_1}{\lambda_1}\frac{a}{b},\frac{\mu_2}{\lambda_2}\frac{a}{b})$.
The global existence of classical solutions and
the stability of the above equilibrium solution of \eqref{Main Intro-eq0}+\eqref{Main Intro-eq1} are among central dynamical issues.
 They have been studied in many papers (see \cite{EEspejoand-TSuzuki, Horstman, Jin,Lin_Mu_Gao, JLZAW, QZhanYLi, PLJSZW, MLACLE, YWang1, WangXiang, PCHu1} and the references therein). For example, in \cite{QZhanYLi}, amount others, the authors proved that

 \smallskip

\noindent $\bullet$ {\it  If $b>\chi_{1}\mu_1-\chi_2\mu_2$, or $N\leq 2$, or $\frac{N-2}{N}(\chi_1\mu_1-\chi_2\mu_2)<b$ and $N\geq 3$,
then for every nonnegative initial $u_0\in C^{0}(\overline{\Omega}),$ \eqref{Main Intro-eq0}+\eqref{Main Intro-eq1} has a unique global classical solution $(u(\cdot,\cdot),v_1(\cdot,\cdot),v_2(\cdot,\cdot))$ which is uniformly bounded.
}
\smallskip

\noindent $\bullet$ {\it If $a=b>2\chi_1\mu_1$, then for every nonnegative initial $u_0\in C^{0}(\overline{\Omega}),$ $u_0\neq 0$, the global classical solution $(u(\cdot,\cdot),v_1(\cdot,\cdot),v_2(\cdot,\cdot))$ of \eqref{Main Intro-eq0}+\eqref{Main Intro-eq1} satisfies $$\lim_{t\to\infty}\Big[\|u(\cdot,t)-1\|_{C^{0}(\Omega)}+\|v_1(\cdot,t)-\frac{\mu_1}{\lambda_1}\|_{C^{0}(\Omega)}+\|v_2(\cdot,t)-\frac{\mu_2}{\lambda_2}\|_{C^{0}(\Omega)} \Big] =0.$$
}

  While  attraction-repulsion chemotaxis systems on bounded domains have been studied in many papers,
 there is little study of such systems on unbounded domains.
The objective of this paper is  to study the dynamics of \eqref{Main Intro-eq0} with $\Omega=\R^N$,
that is,
\begin{equation}\label{Main Intro-eq}
\begin{cases}
u_{t}=\Delta u -\chi_1\nabla( u\nabla v_1)+\chi_2 \nabla(u\nabla v_2 )+ u(a -b u) , \qquad \  x\in\R^N,\ t>0,  \\
0=(\Delta- \lambda_1 I)v_1+ \mu_1 u  , \qquad \ x\in\R^N,\ t>0,  \\
0=(\Delta- \lambda_2 I)v_2+ \mu_2 u  , \qquad \ x\in\R^N,\ t>0,  \\
u(\cdot,0)=u_{0} , \qquad x\in  \mathbb{R}^N .
\end{cases}
\end{equation}
 In the case that the  chemorepellent is absent, that is, $\chi_2=0$,
 the authors of the current paper studied in \cite{SaSh1} the global existence of classical solutions and asymptotic behavior of bounded
 global classical solutions of \eqref{Main Intro-eq}. In the current paper, we  investigate the global existence of classical solutions, stability of constant equilibria, and spreading speeds of \eqref{Main Intro-eq} when both  chemoattractant and  chemorepellent are present.
 More precisely, we  identify the circumstances under which  positive classical solutions of \eqref{Main Intro-eq} with nonnegative, bounded,  and uniformly continuous initial functions  exist globally;  investigate the asymptotic stability of the nonzero constant equilibrium  $(\frac{a}{b},\frac{\mu_1}{\lambda_1}\frac{a}{b},\frac{\mu_2}{\lambda_2}\frac{a}{b})$;  and explore the
 spreading properties of the global solutions with  compactly supported initial functions.
   We pay special attention to the combined effect of the chemoattractant and  chemorepellent on the above dynamical issues.

Note that, due to biological interpretations, only nonnegative initial functions will be of interest. We call $(u(x,t),v_1(x,t),v_2(x,t))$ a {\it  classical solution} of \eqref{Main Intro-eq} on $ [0,T)$ if  $u,v_1,v_2\in C(\R^N\times [0,T))\cap  C^{2,1}(\R^N\times (0,T))$ and satisfies \eqref{Main Intro-eq} for $(x,t)\in\R^N\times (0,T)$ in the classical sense. A classical solution $(u(x,t),v_1(x,t),v_2(x,t))$ of \eqref{Main Intro-eq} on $ [0,T)$ is called {\it nonnegative} if $u(x,t)\ge 0$, $v_1(x,t)\ge 0$ and $v_2(x,t)\geq 0$ for all $(x,t)\in\R^N\times [0,T)$.  A {\it global classical solution}   of \eqref{Main Intro-eq} is a classical solution  on $ [0,\infty)$.

Let
\begin{equation}
\label{unif-cont-space}
C_{\rm unif}^b(\R^N)=\{u\in C(\R^N)\,|\, u(x)\,\,\text{is uniformly continuous in}\,\, x\in\R^N\,\, {\rm and}\,\, \sup_{x\in\R^N}|u(x)|<\infty\}
\vspace{-0.1in}\end{equation}
equipped with the norm $\|u\|_\infty=\sup_{x\in\R^N}|u(x)|$.
We have the following result on the global existence of classical solutions of \eqref{Main Intro-eq} for initial functions  belonging to $C^{b}_{\rm unif}(\R^N)$.

\medskip

\noindent {\bf Theorem A.}
{\it Suppose that
 \begin{equation}\label{gl-exist-thm-eq2'}
\chi_1=a=b=0
\end{equation}
or
\begin{eqnarray}\label{gl-exist-thm-eq2}
b  > \chi_1\mu_1-\chi_2\mu_2 + M,
\end{eqnarray}
where
\begin{align}\label{gl-exist-thm-eq1}
M:= \min\Big\{ &\frac{1}{\lambda_2}\big( (\chi_2\mu_2\lambda_2-\chi_1\mu_1\lambda_1)_+ + \chi_1\mu_1(\lambda_1-\lambda_2)_+ \big),\nonumber\\
& \qquad \frac{1}{\lambda_1}\big( (\chi_2\mu_2\lambda_2-\chi_1\mu_1\lambda_1)_+ + \chi_2\mu_2(\lambda_1-\lambda_2)_+ \big) \Big\}.
 \end{align}
Then for every nonnegative initial function $u_{0}\in C^{b}_{\rm unif}(\R^N)$, \eqref{Main Intro-eq} has a unique nonnegative global classical solution $(u(\cdot,\cdot;u_0),v_1(\cdot,\cdot;u_0),v_2(\cdot,\cdot;u_0))$ with $u(\cdot,0;u_0) =u_0$. Furthermore, it holds that
\begin{equation}
\|u(\cdot,t;u_0)\|_{\infty}\leq\begin{cases}\|u_0\|_{\infty}\hspace{4.2 cm}\text{if \eqref{gl-exist-thm-eq2'} holds}\\
\max\{\|u_0\|_{\infty}, \frac{a}{b+\chi_2\mu_2-\chi_1\mu_1-M}\} \quad \text{if \eqref{gl-exist-thm-eq2} holds}.
\end{cases}
\end{equation}
}

\begin{rk}\label{remark-on-glo-exist}   $M\leq \chi_2\mu_2$.  \eqref{gl-exist-thm-eq2'} and \eqref{gl-exist-thm-eq2}
provide explicit conditions for the global existence of classical solutions. The following special and important conditions follow from \eqref{gl-exist-thm-eq2}.
\begin{description}
\item[(i)] If $b>\chi_1\mu_1$, \eqref{Main Intro-eq} always has global bounded classical solution for any initial $u_0\in C^{b}_{\rm unif}(\R^N)$ with  $u_0\geq 0$.

\item[(ii)] If $\lambda_1\leq \lambda_2$ and $ \chi_2\mu_2\lambda_2\geq \chi_1\mu_1\lambda_1$, we have that $M=\chi_2\mu_2-\frac{\lambda_1}{\lambda_2}\chi_1\mu_1$. In this case, it follows from Theorem A that for every nonnegative bounded and uniformly continuous initial data $u_0$, \eqref{Main Intro-eq} has a unique bounded global classical solution $(u(\cdot,\cdot;u_0),v_1(\cdot,\cdot;u_0),v_2(\cdot,\cdot;u_0))$, whenever $b>\chi_1\mu_1(1-\frac{\lambda_1}{\lambda_2})$. Thus, in the absence of chemoattractant, i.e $\chi_1=0$, for every nonnegative bounded and uniformly continuous initial data $u_0$, \eqref{Main Intro-eq} has a unique bounded global classical solution $(u(\cdot,\cdot;u_0),v_1(\cdot,\cdot;u_0),v_2(\cdot,\cdot;u_0))$, whenever $b>0$.

\item[(iii)]If $\lambda_1\leq \lambda_2$ and $ \chi_2\mu_2\lambda_2\leq \chi_1\mu_1\lambda_1$, we have that $M=0$. In this case, it follows from Theorem A that for every nonnegative bounded and uniformly continuous initial data $u_0$, \eqref{Main Intro-eq} has a unique bounded global classical solution $(u(\cdot,\cdot;u_0),v_1(\cdot,\cdot;u_0),v_2(\cdot,\cdot;u_0))$, whenever $b>\chi_1\mu_1-\chi_2\mu_2$.

\item[(iv)] We note that if $\lambda_1\geq \lambda_2$ and $ \chi_2\mu_2\lambda_2\ge\chi_1\mu_1\lambda_1 $, then  $M= \chi_2\mu_2-\chi_1\mu_1$. Thus, if $\lambda_1\geq \lambda_2$ and $ \chi_2\mu_2\lambda_2\geq\chi_1\mu_1\lambda_1 $, it follows from Theorem A that for every $b>0$ and for every nonnegative bounded and uniformly continuous initial data $u_0$, \eqref{Main Intro-eq} has a unique bounded global classical solution $(u(\cdot,\cdot;u_0),v_1(\cdot,\cdot;u_0),v_2(\cdot,\cdot;u_0))$.

\item[(v)] If $\lambda_1\geq \lambda_2$ and $ \chi_2\mu_2\lambda_2\leq \chi_1\mu_1\lambda_1$, we have that $M=\frac{(\lambda_1-\lambda_2)\chi_2\mu_2}{\lambda_1}$. In this case, it follows from Theorem A that for every nonnegative bounded and uniformly continuous initial data $u_0$, \eqref{Main Intro-eq} has a unique bounded global classical solution $(u(\cdot,\cdot;u_0),v_1(\cdot,\cdot;u_0),v_2(\cdot,\cdot;u_0))$, whenever $b>\chi_1\mu_1-\frac{\lambda_2}{\lambda_1}\chi_2\mu_2$.
\end{description}
\end{rk}

 \medskip

 It follows from Remark \ref{remark-on-glo-exist} (iii)\&(v), that when $\chi_2=0$, we recover as a special case Theorem 1.5 in \cite{SaSh1} for the case $b>\chi_1$ and $\mu_1=1$. When \eqref{gl-exist-thm-eq2} does not hold, we leave it open whether for any  nonnegative initial function $u_{0}\in C^{b}_{\rm unif}(\R^N)$ global solution to \eqref{Main Intro-eq} exists.

  Theorem A is fundamental. Assume the conditions in Theorem A. Then \eqref{Main Intro-eq} generates a dynamical system on the infinite dimensional
 space $X^+=\{u\in C_{\rm unif}^b(\R^N)\,|\, u\ge 0\}$. Methods and theorems for general  infinite dimensional dynamical systems in literature
 (e.g. \cite{Hal}, \cite{SeYu}) may then be utilized for the further study of many important dynamical aspects, including the long time behavior of bounded solutions,  stability of certain special solutions, existence of global attractor, etc. In the following, we explore the stability of the nonzero constant equilibrium  $(\frac{a}{b},\frac{\mu_1}{\lambda_1}\frac{a}{b},\frac{\mu_2}{\lambda_2}\frac{a}{b})$.

 We first study the stability of   $(\frac{a}{b},\frac{\mu_1}{\lambda_1}\frac{a}{b},\frac{\mu_2}{\lambda_2}\frac{a}{b})$ with respect to strictly positive initial functions. From now on, we shall always suppose that $a>0$, unless otherwise specified. We prove

\medskip

\noindent {\bf Theorem B.}
{\it Suppose that
\begin{equation}\label{main-asym-eq}
b>\chi_1\mu_1-\chi_2\mu_2+K,
\end{equation}
where
\begin{align}\label{N-eq}
K:=\min\Big\{&\frac{1}{\lambda_2}\Big(|\chi_1\mu_1\lambda_1-\chi_2\mu_2\lambda_2|+\chi_1\mu_1|\lambda_1-\lambda_2|\Big),\nonumber\\
&\quad  \frac{1}{\lambda_1}\Big(|\chi_1\mu_1\lambda_1-\chi_2\mu_2\lambda_2|+\chi_2\mu_2|\lambda_1-\lambda_2|\Big) \Big\}.
\end{align}
Then for every initial function $u_{0}\in C^{b}_{\rm unif}(\R^N)$ with $\inf_{x\in\R^N}u_0(x)>0$, \eqref{Main Intro-eq} has a unique bounded global classical solution $(u(\cdot,\cdot;u_0),v_1(\cdot,\cdot;u_0),v_2(\cdot,\cdot;u_0))$ with $u(\cdot,0;u_0)=u_{0}$. Furthermore we have that
\begin{equation}
\lim_{t\to\infty}\|u(\cdot,t;u_0)-\frac{a}{b}\|_{\infty}=0
\end{equation}
and
\begin{equation}
\lim_{t\to\infty}\|\lambda_iv_i(\cdot,t;u_0)-\frac{a}{b}\mu_i\|_{\infty}=0 , \ \forall\ i=1,2.
\end{equation}

%Suppose that one of the following holds \begin{description}
%\item[(i)] If $\lambda_{1}\leq \lambda_2, \quad  \chi_2\mu_2\lambda_2\geq \chi_1\mu_1\lambda_1$, then \eqref{main-asym-eq} holds if and only if $ b>\frac{2}{\lambda_2}(\lambda_2-\lambda_1)\chi_1\mu_1,$
%\item[(ii)]$\lambda_{1}\leq \lambda_2, \quad  \chi_2\mu_2\lambda_2\leq \chi_1\mu_1\lambda_1,$ and $b+\chi_2\mu_2-\chi_1\mu_1>0$
%\item[(iii)]$\lambda_{1}\geq \lambda_2, \quad \chi_2\mu_2\lambda_2\geq \chi_1\mu_1\lambda_1$, and $b>0$,
%\item[(iv)] $\lambda_{1}\geq \lambda_2, \quad \chi_1\mu_1\lambda_1\geq \chi_2\mu_2\lambda_2, \quad \text{and} \quad  b-\frac{2}{\lambda_1}(\chi_1\mu_1\lambda_1-\chi_2\mu_2\lambda_2)>0,$
%\end{description}
%then for every initial data $u_{0}\in C^{b}_{\rm unif}(\R^N)$ with $\inf_{x\in\R^N}u_0(x)>0$, \eqref{Main Intro-eq} has a unique bounded global classical solution $(u(\cdot,\cdot;u_0),v(\cdot,\cdot;u_0),w(\cdot,\cdot;u_0))$ with $u(\cdot,0;u_0)=u_{0}$. Furthermore we have that
%\begin{equation}\lim_{t\to\infty}\|u(\cdot,t;u_0)-\frac{a}{b}\|_{\infty}=0 \end{equation}and \begin{equation}\lim_{t\to\infty}\|\lambda_iv_i(\cdot,t;u_0)-\frac{a}{b}\mu_i\|_{\infty}=0 , \ \forall\ i=1,2.\end{equation}
}

\medskip

\begin{rk}
\begin{itemize}
\item[(1)]  \eqref{main-asym-eq} provides explicit conditions for the global stability of the constant equilibrium  $(\frac{a}{b},\frac{\mu_1}{\lambda_1}\frac{a}{b},\frac{\mu_2}{\lambda_2}\frac{a}{b})$ with respect to strictly positive initial functions.
We point out the  following special and important equivalent  conditions of \eqref{main-asym-eq}.
\begin{itemize}
\item[(i)] If $\lambda_{1}\leq \lambda_2$, and   $\chi_2\mu_2\lambda_2\geq \chi_1\mu_1\lambda_1$, then \eqref{main-asym-eq} holds if and only if $ b>2 \chi_1\mu_1-2\frac{\lambda_1}{\lambda_2}\chi_1\mu_1$.

\item[(ii)] If $\lambda_{1}\leq \lambda_2$, and  $\chi_2\mu_2\lambda_2\leq \chi_1\mu_1\lambda_1,$   then \eqref{main-asym-eq} holds if and only if $b>2\chi_1\mu_1-2\chi_2\mu_2)$.

\item[(iii)] If $\lambda_{1}\geq \lambda_2$, and   $\chi_2\mu_2\lambda_2\geq \chi_1\mu_1\lambda_1$, then \eqref{main-asym-eq} holds if and only if $b>0$.

\item[(iv)] If $\lambda_{1}\geq \lambda_2$, and $\chi_1\mu_1\lambda_1\geq \chi_2\mu_2\lambda_2$, then \eqref{main-asym-eq} holds if and only if $
b>2 \chi_1\mu_1-2\frac{\lambda_2}{\lambda_1}\chi_2\mu_2$.

\end{itemize}
 \item[(2)]  By (i)-(iv), if $b>2\chi_1\mu_1$, then \eqref{main-asym-eq} holds. Hence the hypothesis \eqref{main-asym-eq} is weaker than the known result on bounded domain.

 \item[(3)] If $\chi_2=0$, then (ii) and  (iv) extend \cite[Theorem 1.7]{SaSh1}.

 \item[(4)] By (i) and (iii), if $\chi_1=0$,  then the constant solution $\frac{a}{b}$ is stable with respect to strictly positive perturbation whenever $b>0$.

 \item[(5)] It is interesting to know whether hypothesis \eqref{gl-exist-thm-eq2} is enough to have the stability of the constant steady solution $(\frac{a}{b},\frac{a\mu_1}{b\lambda_1},\frac{a\mu_2}{b\lambda_2})$ with respect to strictly positive perturbation. We plan to study this question in our future work.
 \end{itemize}
\end{rk}

\medskip

 Next, we study the attraction   of  $(\frac{a}{b},\frac{\mu_1}{\lambda_1}\frac{a}{b},\frac{\mu_2}{\lambda_2}\frac{a}{b})$ with respect to global classical solutions of \eqref{Main Intro-eq} with compactly supported initial
functions, or equivalently, the spreading properties of  global classical solutions of \eqref{Main Intro-eq} with compactly supported initial
functions.  {For $x=(x_1,x_2,\cdots,x_N)\in\mathbb{R}^N$, let $|x|=\big(\sum_{i=1}^N x_i^2\big)^{\frac{1}{2}}$}. We obtain the following main results.

\medskip

\noindent{\bf Theorem C.}
{\it Suppose that \eqref{gl-exist-thm-eq2} holds and define
\begin{equation}\label{asym-hyp1}
D:=\min\Big\{\frac{|\chi_2\mu_2-\chi_1\mu_1|}{2\sqrt{\lambda_2}}+\frac{\chi_1\mu_1|\sqrt{\lambda_1}-\sqrt{\lambda_2}|}{2\sqrt{\lambda_1\lambda_2}}, \frac{|\chi_1\mu_1-\chi_2\mu_2|}{2\sqrt{\lambda_1}}+\frac{\chi_2\mu_2|\sqrt{\lambda_2}-\sqrt{\lambda_1}|}{2\sqrt{\lambda_1\lambda_2}}\Big\}.
\end{equation}
Then  for every $u_{0}\in C^{b}_{\rm unif}(\R^N)$ with $u_0\geq 0$ and $supp(u_0)$ being compact and non-empty, we have that
\begin{equation}\label{asym-main-eq1}
\lim_{t\to\infty}\left[\sup_{|x|\geq ct}|u(x,t;u_0)| + \sup_{|x|\geq ct}| v_1(x,t;u_0)| + \sup_{|x|\geq ct}|v_2(x,t;u_0)|\right]=0
\end{equation}
for every $c> c_+^{*}(\chi_1,\mu_1,\lambda_1,\chi_2,\mu_2,\lambda_2)$, where
\begin{equation}
c_{+}^{*}(\chi_1,\mu_1,\lambda_1,\chi_2,\mu_2,\lambda_2)=2\sqrt{a}+\frac{\sqrt{a}(D\sqrt{Na}+\chi_2\mu_2)}{b+\chi_2\mu_2-\chi_1\mu_1-M},
\end{equation}
and $M$ is given by \eqref{gl-exist-thm-eq1}.
}

\medskip

\begin{rk}
\begin{description}
\item[(i)]If $\lambda_1\leq \lambda_2$ and $\chi_2\mu_2\lambda_2\geq \chi_1\mu_1\lambda_1$, then $$
c_{+}^{*}(\chi_1,\mu_1,\lambda_1,\chi_2,\mu_2,\lambda_2)=2\sqrt{a}+\frac{\sqrt{a}(D\sqrt{Na}+\chi_2\mu_2)}{b-(1-\frac{\lambda_1}{\lambda_2})\chi_1\mu_1}.
$$
\item[(ii)] If $\lambda_1\leq \lambda_2$ and $\chi_2\mu_2\lambda_2\leq \chi_1\mu_1\lambda_1$, then $$
c^{*}_+(\chi_1,\mu_1,\lambda_1,\chi_1,\mu_2,\lambda_2)=2\sqrt{a}+\frac{\sqrt{a}(D\sqrt{Na}+\chi_2\mu_2)}{b+\chi_2\mu_2-\chi_1\mu_1}.
 $$
\item[(iii)] If $\lambda_1\geq \lambda_2 $ and $\chi_{2}\mu_2\lambda_2\geq \chi_1\mu_1\lambda_1$ then
$$
c^{*}_+(\chi_1,\mu_1,\lambda_1,\chi_1,\mu_2,\lambda_2)=2\sqrt{a}+\frac{\sqrt{a}(D\sqrt{Na}+\chi_2\mu_2)}{b}.
$$
\item[(iv)] If $\lambda_1\geq \lambda_2 $ and $\chi_{2}\mu_2\lambda_2\leq \chi_1\mu_1\lambda_1$, then
$$
c^{*}_+(\chi_1,\mu_1,\lambda_1,\chi_1,\mu_2,\lambda_2)=2\sqrt{a}+\frac{\sqrt{a}(D\sqrt{Na}+\chi_2\mu_2)}{b-\frac{1}{\lambda_1}(\chi_1\mu_1\lambda_1-\chi_2\mu_2\lambda_2) }.$$
\item[(v)] Note that  $\chi_2=0$ implies that $D=\frac{\chi_1\mu_1}{2\sqrt{\lambda_1}}$ and $M=0$. Hence if $\chi_2=0$, it follows from Theorem C that $c^{*}_{+}(\chi_1,\mu_1,\lambda_1,0,\mu_2,\lambda_2)=2\sqrt{a}+\frac{a\chi_1\mu_1\sqrt{N}}{2(b-\chi_1\mu_1)\sqrt{\lambda_1}}$. Thus, in the case $\chi_2=0$, and $\mu_1=\lambda_1=1$, we obtain a better estimate for $c^{*}_{+}(\chi_1,\mu_1,\lambda_1,0,\mu_2,\lambda_2)$ compare to the one giving by \cite[Remark 1.2(iii)]{SaSh2}.
\end{description}
\end{rk}

\medskip

\noindent {\bf Theorem D.}
{\it
Suppose that \eqref{main-asym-eq} holds and
\begin{align}\label{asym-hyp2}
4a(1-L)-\frac{Na^2D^2}{(b+\chi_2\mu_2-\chi_1\mu_1-M)^2}>0,
\end{align}
where $M$ is given by \eqref{gl-exist-thm-eq1} and
 \begin{align}\label{L-eq}
L:=\min\Big\{&\frac{(\chi_2\mu_2\lambda_2-\chi_1\mu_1\lambda_1)_{-}+\chi_1\mu_1(\lambda_1-\lambda_2)_{-}}{\lambda_2(b+\chi_2\mu_2-\chi_1\mu_1-M)},
\nonumber\\
&\quad  \frac{(\chi_2\mu_2\lambda_2-\chi_1\mu_1\lambda_1)_{-}+\chi_2\mu_2(\lambda_1-\lambda_2)_{-}}{\lambda_1(b+\chi_2\mu_2-\chi_1\mu_1-M)}  \Big\}.
\end{align}
Then for every $u_{0}\in C^{b}_{\rm unif}(\R^N)$ with  $u_0\geq 0$ and $supp(u_0)$ being  non-empty, we have that
\begin{equation}\label{asym-main-eq2}
\lim_{t\to\infty}\left[\sup_{|x|\leq ct}|u(x,t;u_0)-\frac{a}{b}| + \sup_{|x|\leq ct}|\lambda_1 v_1(x,t;u_0)-\frac{a}{b}\mu_1| + \sup_{|x|\leq ct}|\lambda_2 v_2(x,t;u_0)-\frac{a}{b}\mu_2|\right]=0
\end{equation}
for every $0\leq c< c_{-}^{*}(\chi_1,\mu_1,\lambda_1,\chi_2,\mu_2,\lambda_2)$, where
 $$
c^*_{-}(\chi_1,\mu_1,\lambda_1,\chi_2,\mu_2,\lambda_2)=2\sqrt{a(1-L)}-\frac{aD\sqrt{N}}{b+\chi_2\mu_2-\chi_1\mu_1-M}.
 $$
}

\medskip

\begin{rk}
\begin{description}
\item[(i)] If $\lambda_1\leq \lambda_2$ and $\chi_2\mu_2\lambda_2\geq\chi_1\mu_1\lambda_1$, then $L=\frac{\chi_1\mu_1(1-\frac{\lambda_1}{\lambda_2})}{b-\chi_1\mu_1(1-\frac{\lambda_1}{\lambda_2})}$ and
$$c^{*}_{-}(\chi_1,\mu_1,\lambda_1,\chi_2,\mu_2,\lambda_2)= 2\sqrt{\frac{a(b-2\chi_1\mu_1(1-\frac{\lambda_1}{\lambda_2}))}{b-\chi_1\mu_1(1-\frac{\lambda_1}{\lambda_2})}} -\frac{aD\sqrt{N}}{b-\chi_1\mu_1(1-\frac{\lambda_1}{\lambda_2})}.$$

\item[(ii)]If $\lambda_1\leq \lambda_2$ and $\chi_2\mu_2\lambda_2\leq\chi_1\mu_1\lambda_1$, then $L=\frac{\chi_1\mu_1-\chi_2\mu_2}{b+\chi_2\mu_2-\chi_1\mu_1}$ and
$$c^{*}_{-}(\chi_1,\mu_1,\lambda_1,\chi_2,\mu_2,\lambda_2)=2\sqrt{\frac{a(b-2(\chi_1\mu_1-\chi_2\mu_2))}{{b+\chi_2\mu_2-\chi_1\mu_1}}} -\frac{aD\sqrt{N}}{{b+\chi_2\mu_2-\chi_1\mu_1}}.$$
\item[(iii)]
If $\lambda_1\geq \lambda_2$ and $\chi_2\mu_2\lambda_2\geq\chi_1\mu_1\lambda_1$, then $L=0$ and
$$c^{*}_{-}(\chi_1,\mu_1,\lambda_1,\chi_2,\mu_2,\lambda_2)= 2\sqrt{a} -\frac{aD\sqrt{N}}{b}.$$

\item[(iv)]If $\lambda_1\geq \lambda_2$ and $\chi_2\mu_2\lambda_2\leq\chi_1\mu_1\lambda_1$, then $L=\frac{\chi_1\mu_1\lambda_1-\chi_2\mu_2\lambda_2}{\lambda_1(b+\chi_2\mu_2\lambda_2-\chi_1\mu_1\lambda_1)}$ and
$$c^{*}_{-}(\chi_1,\mu_1,\lambda_1,\chi_2,\mu_2,\lambda_2)= 2\sqrt{\frac{a(b-\frac{2}{\lambda_1}(\chi_1\mu_1\lambda_1-\chi_2\mu_2\lambda_2))}{{b-\frac{1}{\lambda_1}(\chi_1\mu_1\lambda_1-\chi_2\mu_2\lambda_2)}}} -\frac{aD\sqrt{N}}{{b-\frac{1}{\lambda_1}(\chi_1\mu_1\lambda_1-\chi_2\mu_2\lambda_2)}}.$$

\item[(v)] If $\chi_2=0$, by $(ii)$ and $(iv)$, we have  that $c^{*}_{-}(\chi_1,\mu_1,\lambda_1,\chi_2,\mu_2,\lambda_2)= 2\sqrt{\frac{a(b-2\chi_1\mu_1)}{b-\chi_1\mu_1}}-\frac{a\chi_1\mu_1\sqrt{N}}{2(b-\chi_1\mu_1)\sqrt{\lambda_1}}.$ Hence in the case $\chi_2=0,$ $\mu_1=\lambda_1=1$, we obtain a better estimate on $c^{*}_{-}(\chi_1,\mu_1,\lambda_1,\chi_2,\mu_2,\lambda_2)$ than the ones obtained in \cite{SaSh2} and \cite{SaSh1}.
\end{description}
\end{rk}

Observe that, if either $\chi_1=\chi_2=0$ or $\chi_1-\chi_2=\mu_1-\mu_2=\lambda_1-\lambda_2=0$, the first equation in \eqref{Main Intro-eq} becomes the following  scalar reaction diffusion equation,
\begin{equation}
\label{fisher-eq}
u_{t}=\Delta u + u(a-bu),\quad  x\in\R^N,\,\, t>0,
\end{equation}
which is referred to as Fisher or KPP equations due to  the pioneering works by Fisher (\cite{Fis}) and Kolmogorov, Petrowsky, Piscunov
(\cite{KPP}) on the spreading properties of \eqref{fisher-eq}.
It follows from the works \cite{Fis}, \cite{KPP}, and \cite{Wei1}  that $c^*_{-}$ and $c^*_{+}$  in Theorem C and Theorem D,  respectively, can be chosen so that $c^{\ast}_{-}=c^{\ast}_{+}=2\sqrt a$
 ($c^*:=2\sqrt a$ is called the {\it spatial spreading speed} of \eqref{fisher-eq} in literature), and that \eqref{fisher-eq} has traveling wave solutions $u(t,x)=\phi(x-ct)$  connecting $\frac{a}{b}$ and $0$ (i.e.
$(\phi(-\infty)=\frac{a}{b},\phi(\infty)=0)$) for all speeds $c\geq c^*$ and has no such traveling wave
solutions of slower speed.
 Since the pioneering works by  Fisher \cite{Fis} and Kolmogorov, Petrowsky,
Piscunov \cite{KPP},  a huge amount research has been carried out toward the spreading properties of
  reaction diffusion equations of the form,
\begin{equation}
\label{general-fisher-eq}
u_t=\Delta u+u f(t,x,u),\quad x\in\R^N,
\end{equation}
where $f(t,x,u)<0$ for $u\gg 1$,  $\p_u f(t,x,u)<0$ for $u\ge 0$ (see \cite{Berestycki1, BeHaNa1, BeHaNa2, Henri1, Fre, FrGa, LiZh, LiZh1, Nad, NoRuXi, NoXi1, She1, She2, Wei1, Wei2, Zla}, etc.).

\begin{rk}
\begin{description}
\item[(i)]{ It is clear from Theorem C and Theorem D that
$$\lim_{(\chi_1,\chi_2)\to (0,0)}c^*_-(\chi_1,\mu_1,\lambda_1,\chi_2,\mu_2,\lambda_2)=\lim_{(\chi_1,\chi_2)\to (0,0)}c^*_+(\chi_1,\mu_1,\lambda_1,\chi_2,\mu_2,\lambda_2)=2\sqrt{a}$$ and
\begin{align*}
 &\lim_{(\delta_1,\delta_2,\delta_3)\to (0,0,0)}c^*_-(\chi+\delta_1,\mu+\delta_2,\lambda+\delta_3,\chi,\mu,\lambda)\nonumber\\
 &=\lim_{(\delta_1,\delta_2,\delta_3)\to (0,0,0)}c^*_+(\chi+\delta_1,\mu+\delta_2,\lambda+\delta_3,\chi,\mu,\lambda)\nonumber\\
  &\lim_{(\delta_1,\delta_2,\delta_3)\to (0,0,0)}c^*_-(\chi,\mu,\lambda,\chi+\delta_1,\mu+\delta_2,\lambda+\delta_3)\nonumber\\
 &=\lim_{(\delta_1,\delta_2,\delta_3)\to (0,0,0)}c^*_+(\chi,\mu,\lambda,\chi+\delta_1,\mu+\delta_2,\lambda+\delta_3)\nonumber\\
& =2\sqrt{a}, \quad \forall  \ \chi>0,\mu>0\  \text{and} \ \lambda>0.
 \end{align*}

 Hence we recover the know results in the literature when $\chi_1=\chi_2=0$ or $\chi_1-\chi_2=\mu_1-\mu_2=\lambda_1-\lambda_2=0$.
}
\item[(ii)] For every $\chi_i\geq 0,\ \mu_i>0,\ \lambda_i>0$, let
$$c^*_{up}(\chi_1,\mu_1,\lambda_1,\chi_2,\mu_2,\lambda_2)=\inf\{ c^*>0\ | \  \eqref{asym-main-eq1} \ \text{holds}\} $$
 and
 $$c^*_{low}(\chi_1,\mu_1,\lambda_1,\chi_2,\mu_2,\lambda_2)=\sup\{ c^*\geq 0 \ | \  \eqref{asym-main-eq2}\ \text{holds}\}.
  $$
  $[c^*_{low}(\chi_1,\mu_1,\lambda_1,\chi_2,\mu_2,\lambda_2),c^*_{up}(\chi_1,\mu_1,\lambda_1,\chi_2,\mu_2,\lambda_2)]$ is called {\rm the spreading speed interval} of \eqref{Main Intro-eq}.
  Theorem C implies that if   \eqref{gl-exist-thm-eq2} holds, then
  $$c^*_{up}(\chi_1,\mu_1,\lambda_1,\chi_2,\mu_2,\lambda_2)\leq  c^*_{+}(\chi_1,\mu_1,\lambda_1,\chi_2,\mu_2,\lambda_2)<\infty.
   $$
   Under the hypotheses of Theorem D, we have that
   $$c^*_{low}(\chi_1,\mu_1,\lambda_1,\chi_2,\mu_2,\lambda_2)\geq c^*_{-}(\chi_1,\mu_1,\lambda_1,\chi_2,\mu_2,\lambda_2)>0.
    $$
    It is interesting  to know the relationship between $c^*_{up}(\chi_1,\mu_1,\lambda_1,\chi_2,\mu_2,\lambda_2)$ and $2\sqrt{a}$ as well as the relationship between $c^*_{low}(\chi_1,\mu_1,\lambda_1,\chi_2,\mu_2,\lambda_2)$ and $2\sqrt{a}$. It is also interesting
      to know whether $c^*_{low}(\chi_1,\mu_1,\lambda_1,\chi_2,\mu_2,\lambda_2)=c^*_{up}(\chi_1,\mu_1,\lambda_1,\chi_2,\mu_2,\lambda_2)$.
       We plan to study these questions in our future works.

\item[(iii)] When $\chi_2=0$, $\lambda_1=\mu_1=1$,  and $0< \chi_1<\frac{b}{2}$, in a very recent work  \cite{SaSh2} it was shown  that there is a positive constant $c^{*}(\chi_1)\geq 2\sqrt{a}$ such that for every $c\geq c^{*}(\chi_1)$ and $\xi\in S^{N-1}$,
    \eqref{Main Intro-eq} has a traveling wave solution $(u(x,t),v(x,t))=(u(x\cdot\xi-ct),v(x\cdot\xi-ct))$ connecting the trivial solutions  $(\frac{a}{b},\frac{a}{b})$ and $(0,0)$ and propagating in the direction of $\xi$ with speed $c$, and  no such traveling wave solution exists for speed less than $2\sqrt{a}$. We plan to study these questions for \eqref{Main Intro-eq} when both $\chi_1>0$ and $\chi_2>0$.
\end{description}
\end{rk}

\medskip

    We end up the introduction with the following remarks.
    First, our study is based  on many techniques developed in \cite{SaSh1}. But, to apply these techniques to \eqref{Main Intro-eq}
   with non-zeros $\chi_1$ and $\chi_2$,
    nontrivial modifications  are
   needed and made in the current paper. The modified techniques
   would be useful for the further study of attraction-repulsion chemotaxis systems. Second, most results obtained in \cite{SaSh1} for the
   special case $\chi_2=0$ are recovered and extended further in the current paper. Third, conditions explicitly depending on the sensitivity parameters $\chi_1$ and $\chi_2$ of the chemoattractant and  chemorepellent are provided in the current paper for the global existence of classical solutions of \eqref{Main Intro-eq} and stability of the nonzero constant equilibrium $(\frac{a}{b},\frac{\mu_1}{\lambda_1}\frac{a}{b},\frac{\mu_2}{\lambda_2}\frac{a}{b})$, and lower and upper bounds explicitly depending
     on $\chi_1$ and $\chi_2$ are established for the spreading speeds of positive solutions with compactly supported initial distributions.
   These conditions and lower and upper bounds would be of great practical importance.

 The rest of the paper is organized as follows. Section 2 is devoted to the study of global existence of classical solutions. It is here that we prove Theorem A. In section 3,  we study the asymptotic stability of the constant equilibrium
$(\frac{a}{b},\frac{\mu_1}{\lambda_1}\frac{a}{b},\frac{\mu_2}{\lambda_2}\frac{a}{b})$ and prove Theorem B.
We study the spreading properties of global classical solutions of \eqref{Main Intro-eq} with compactly supported initial
functions and prove Theorems C and D in section 4.

\medskip

%\noindent {\bf Acknowledgment.} The authors would like to thank  the referee for valuable comments and suggestions which improved the presentation of this paper considerably.

\section{Global Existence}
In this section, we discuss  the existence of global/bounded  classical solutions and prove Theorem A. We start with the following result which guarantees the existence of a unique local in time classical solution of \eqref{Main Intro-eq} for any nonnegative bounded and uniformly continuous initial data.

\begin{lem}\label{Local Existence}
For any $u_0 \in C_{\rm unif}^{b}(\R^N)$ with  $u_0 \geq 0$,  there exists $T_{max} \in (0,\infty]$  such that \eqref{Main Intro-eq}  has a unique non-negative classical solution $(u(x,t;u_0),v_1(x,t;u_0), v_2(x,t,u_0))$ { on $[0,T_{\max})$} with $\lim_{t\to 0}u(\cdot,t;u_0)=u_0$ in $C_{\rm uinf}^b(\R^N)$-norm. Moreover, if $T_{max}< \infty,$ then
\begin{equation}\label{Eq_Local Existence}
\limsup_{t \to T_{max}}\| u(\cdot,t;u_0)\|_{\infty}=\infty.
\end{equation}
\end{lem}
\begin{proof}
It follows from the similar arguments used in the proof of \cite[Theorem 1.1]{SaSh1}.
\end{proof}

\begin{proof}[Proof of Theorem A]
Let $u_{0}\in C^{b}_{\rm unif}(\R^N)$ with $u_0\ge 0$ be given and let  $(u(\cdot,\cdot;u_0),v_1(\cdot,\cdot;u_0)$, $v_2(\cdot,\cdot;u_0))$ be the classical solution of \eqref{Main Intro-eq} with initial function $u_0$ defined on the maximal interval $[0, T_{\max})$ of existence. Then,
\begin{eqnarray}\label{global-exist-eq001}
u_{t}&=&\Delta u-\chi_1\nabla(u\nabla v_1)+\chi_2\nabla(u\nabla v_2) +u(a-bu)\nonumber\\
& =& \Delta u +\nabla(\chi_2 v_2-\chi_1 v_1)\nabla u+u(a-\chi_1\Delta v_1 +\chi_2\Delta v_2 -bu),\quad x\in\R^N.
\end{eqnarray}
The second and third equations of \eqref{Main Intro-eq} yield that $\Delta v_i=\lambda_i v_i-\mu_i u$, $i=1,2$. Hence equation \eqref{global-exist-eq001} becomes
\begin{equation}\label{global-exist-eq002}
u_{t}= \Delta u +\nabla(\chi_2 v_2-\chi_1 v_1)\nabla u+u\Big(a+ (\chi_{2}\lambda_2v_2-\chi_1\lambda_1v_1)-(b+\chi_2\mu_2-\chi_1\mu_1 )u\Big),\,\,\, x\in\R^N.
 \end{equation}
 Let
\begin{equation}
C_0:=\begin{cases}
\|u_0\|_{\infty}\hspace{4.2 cm}\text{if }\ \chi_1=a=b=0,\cr
\max\{\|u_0\|_{\infty}, \frac{a}{b+\chi_2\mu_2-\chi_1\mu_1-M}\}\quad \text{if} \ b+\chi_2\mu_2-\chi_1\mu_1-M>0
\end{cases}
\end{equation}
where  $M$ is given by \eqref{gl-exist-thm-eq1}. {Let $T>0$ be a given positive real number and consider $\mathcal{E}^{T}:=C^{b}_{\rm unif}(\R^N\times[0, T])$ endowed with the norm
\begin{equation}\label{global-exist-eq-r001}
\|u\|_{\mathcal{E}^T}:=\sum_{k=1}^{\infty}\frac{1}{2^k}\|u\|_{L^{\infty}([-k,k]\times[0,T])}.
\end{equation}
We note that the convergence in $(\mathcal{E}^T,\|.\|_{\mathcal{E}^T})$ is equivalent to the uniform convergence on compact subsets on $\R^N\times[0,T]$.  Next, we consider the subset $\mathcal{E}$ of $\mathcal{E}^T$ defined by }
 $$\mathcal{E}:=\{u\in  C_{\rm unif}^b(\R^N\times[0, T])\,|\, u(\cdot,0)=u_0, 0\leq u(x,t)\leq C_0, x\in\R^N, 0\leq t\leq T\}.
  $$
{ It is clear that
\begin{equation}\label{global-exist-eq-r002}
\|u\|_{\mathcal{E}^T}\leq C_{0},  \quad \forall\ u\in\mathcal{E}.
\end{equation} It readily follows from the definition of $\mathcal{E}$ and \eqref{global-exist-eq-r002} that $\mathcal{E}$ is a closed bounded and convex subset of $\mathcal{E}^T$.} We shall show that $u(\cdot,\cdot;u_0)\in \mathcal{E}$.

  For every $u\in \mathcal{E}$ let us define $v_{i}(\cdot,\cdot;u)$, $i=1,2$ by
\begin{equation}\label{global-exist-eq003}
v_{i}(x,t;u)=\mu_{i}\int_{0}^{\infty}\int_{\R^N}\frac{e^{-\lambda_i s}}{(4\pi s)^{\frac{N}{2}}}e^{-\frac{|z-x|^2}{4s}}u(z,t)dzds, \quad x\in\R^N,\ t\in[0, T].
\end{equation} and let $U(x,t;u)$ be the solution of the initial value problem
\begin{equation}\label{global-exist-eq004}
\begin{cases}
U_{t}=\Delta U+\nabla\big(\chi_2 v_2(x,t;u)-\chi_1v_1(x,t;u)\big)\nabla U\\
\qquad\,\,\,  +U\Big(a+(\chi_2\lambda_2v_2(x,t;u) -\chi_1\lambda_1v_1(x,t;u)) -(b+\chi_2\mu_2-\chi_1\mu_1)U\Big ),\quad x\in\R^N\\
U(\cdot,0,u)=u_0(\cdot).
\end{cases}
\end{equation}
For every $u\in\mathcal{E}$, using \eqref{global-exist-eq003}, we have that
\begin{eqnarray}\label{global-exist-eq005}
( \chi_2\lambda_2v_2 -\chi_1\lambda_1v_1)(x,t;u) &= & \int_{0}^{\infty}\int_{\R^N}\left[ \chi_2\lambda_2\mu_2e^{-\lambda_{2}s} - \chi_1\lambda_1\mu_1e^{-\lambda_{1}s}\right]\frac{e^{-\frac{|x-z|^2}{4s}}}{(4\pi s)^{\frac{N}{2}}}u(z,t)dzds \nonumber\\
&= & (\chi_2\mu_2\lambda_2-\chi_1\mu_1\lambda_1)\int_0^{\infty}\int_{\R^N}e^{-\lambda_2 s}\frac{e^{-\frac{|x-z|^2}{4s}}}{(4\pi s)^{\frac{N}{2}}}u(z,t)dzds  \nonumber\\
& & +\chi_1\mu_1\lambda_1 \int_{0}^{\infty}\int_{\R^N}(e^{-\lambda_{2}s}-e^{-\lambda_{1}s})\frac{e^{-\frac{|x-z|^2}{4s}}}{(4\pi s)^{\frac{N}{2}}}u(z,t)dzds \nonumber\\
& \leq & (\chi_2\lambda_2\mu_2-\chi_1\lambda_1\mu_1)_{+}C_0\int_{0}^{\infty}\int_{\R^N}e^{-\lambda_{2}s}\frac{e^{-\frac{|x-z|^2}{4s}}}{(4\pi s)^{\frac{N}{2}}}dzds\nonumber\\
& & + \chi_1\mu_1\lambda_1 C_{0} \int_{0}^{\infty}\int_{\R^N}(e^{-\lambda_{2}s}-e^{-\lambda_{1}s})_{+}\frac{e^{-\frac{|x-z|^2}{4s}}}{(4\pi s)^{\frac{N}{2}}}dzds \nonumber\\
&=  &\frac{C_0}{\lambda_2}\Big((\chi_2\lambda_2\mu_2-\chi_1\lambda_1\mu_1)_{+}+\chi_1\mu_1(\lambda_1-\lambda_2)_{+}\Big).
\end{eqnarray}
Similarly, we have that
\begin{eqnarray}\label{global-exist-eq006}
( \chi_2\lambda_2v_2 -\chi_1\lambda_1v_1)(x,t;u) &= & \int_{0}^{\infty}\int_{\R^N}\left[ \chi_2\lambda_2\mu_2e^{-\lambda_{2}s} - \chi_1\lambda_1\mu_1e^{-\lambda_{1}s}\right]\frac{e^{-\frac{|x-z|^2}{4s}}}{(4\pi s)^{\frac{N}{2}}}u(z,t)dzds \nonumber\\
&= &\chi_2\mu_2\lambda_2\int_0^{\infty}\int_{\R^N}(e^{-\lambda_2 s}-e^{-\lambda_1 s})\frac{e^{-\frac{|x-z|^2}{4s}}}{(4\pi s)^{\frac{N}{2}}}u(z,t)dzds  \nonumber\\
& & +(\chi_2\mu_2\lambda_2-\chi_1\mu_1\lambda_1)\int_{0}^{\infty}\int_{\R^N}e^{-\lambda_{1}s}\frac{e^{-\frac{|x-z|^2}{4s}}}{(4\pi s)^{\frac{N}{2}}}u(z,t)dzds \nonumber\\
&\leq & \chi_2\mu_2\lambda_2 C_0\int_0^{\infty}\int_{\R^N}(e^{-\lambda_2 s}-e^{-\lambda_1 s})_{+}\frac{e^{-\frac{|x-z|^2}{4s}}}{(4\pi s)^{\frac{N}{2}}}dzds  \nonumber\\
& & +(\chi_2\mu_2\lambda_2-\chi_1\mu_1\lambda_1)_{+}C_{0}\int_{0}^{\infty}\int_{\R^N}e^{-\lambda_{1}s}\frac{e^{-\frac{|x-z|^2}{4s}}}{(4\pi s)^{\frac{N}{2}}}dzds \nonumber\\
& = & \frac{C_0}{\lambda_{1}}\Big( \chi_2\mu_2(\lambda_1-\lambda_2)_{+} + (\chi_2\mu_2\lambda_2-\chi_1\mu_1\lambda_1)_{+} \Big).
\end{eqnarray}
Thus, it follows from  \eqref{global-exist-eq005} and \eqref{global-exist-eq006} that for every $u\in\mathcal{E}$, we have that
\begin{equation}\label{global-exist-eq007}
(\chi_2 \lambda_2 v_2-\chi_1 \lambda_1 v_1)(x,t;u)\leq M C_{0}
\end{equation}
where $M$ is given by \eqref{gl-exist-thm-eq1}. Thus for every $u\in\mathcal{E}$, we have that
\begin{align}
U_t(x,t;u)\leq & \Delta U(x,t;u)+\nabla(\chi_2v_2-\chi_1v_1)\nabla U(x,t;u)\nonumber\\
& +\underbrace{\Big(a+M C_0-(b+\chi_2\mu_2-\chi_1\mu_1)U(x,t;u)\Big)U(x,t;u)}_{\mathcal{L}(U)}.
\end{align}
Note that
$$
\mathcal{L}(C_0)=\Big(a-(b+\chi_2\mu_2-\chi_1\mu_1-M)C_0\Big)C_0\leq0.
$$
Thus, using comparison principle for parabolic equations, we obtain that
\begin{equation}\label{global-exist-eq008}
U(x,t;u)\leq C_0,\quad \forall\ x\in\R^N,\ \forall\ t\in[0,T],\ \forall\ u\in\mathcal{E}.
\end{equation}
Thus $ U(\cdot,\cdot;u)\in \mathcal{E}$ for every $u\in\mathcal{E}$. { By the arguments in \cite[Lemma 4.3]{SaSh2}, the mapping $\mathcal{E}\ni u\mapsto U(\cdot,\cdot;u)\in\mathcal{E}$ is continuous and compact, and then by Schauder's fixed theorem, it has a fixed point $u^*$}. Clearly $(u^*,v_1(\cdot,\cdot;u^*),v_2(\cdot,\cdot;u^*))$ is a classical solution of \eqref{Main Intro-eq}. Thus, by Lemma \ref{Local Existence}, we have that $T_{\max}\geq T$ and $u(\cdot,\cdot;u_0)=u^*$. Since $T>0$ is arbitrary chosen, Theorem A follows.
\end{proof}

\section{Asymptotic stability of the constant equilibrium $(\frac{a}{b},\frac{\mu_1}{\lambda_1}\frac{a}{b},\frac{\mu_2}{\lambda_2}\frac{a}{b})$ }
In this section, we discuss the asymptotic stability of the constant equilibrium $(\frac{a}{b},\frac{\mu_1}{\lambda_1}\frac{a}{b},\frac{\mu_2}{\lambda_2}\frac{a}{b})$  of \eqref{Main Intro-eq} and prove Theorem B. Throughout this section we suppose that \eqref{gl-exist-thm-eq2} holds, so that for every nonnegative, bounded, and uniformly continuous initial function $u_0$, \eqref{Main Intro-eq} has a nonnegative bounded global classical solution
 $(u(x,t;u_0),v_1(x,t;u_0),v_2(x,t;u_0))$.

For given $u_0\in C^{b}_{\rm unif}(\R^N)$ with $u_0\geq 0 $,  define
$$
\underline{u}:=\liminf_{t\to\infty}\inf_{x\in\R^N}u(x,t;u_0)\quad \text{and}\quad \overline{u}:=\limsup_{t\to\infty}\sup_{x\in\R^N}u(x,t;u_0).
$$
Using the definition of limsup and liminf, we have that for every $\varepsilon>0$, there is $T_{\varepsilon}>0$ such that
$$
\underline{u}-\varepsilon\leq u(x,t;u_0)\leq \overline{u}+\varepsilon\quad \forall\ x\in\R^N,\ \forall\ t\geq T_{\varepsilon}.
$$
Hence, it follows from comparison principle for  elliptic equations, that
\begin{equation}\label{asym-eq03}
\mu_i(\underline{u}-\varepsilon)\leq \lambda_iv_i(x,t;u_0)\leq \mu_i(\overline{u}+\varepsilon), \forall\ x\in\R^N,\ \ \forall \ t\ge T_{\varepsilon},\ i=1,2.
\end{equation}
%Combining  inequalities \eqref{global-exist-eq002} and \eqref{asym-eq03}, we obtain that  $$ u_{t}\leq \Delta u+\nabla(\chi_2v_2-\chi_1v_1)\nabla u+(a+\chi_2\mu_2(\overline{u}+\varepsilon)-\chi_1\mu_1(\underline{u}-\varepsilon) -(b+\chi_2\mu_2-\chi_1\mu_1)u)u\ \ \forall\ t\ge T_{\varepsilon}. $$

We first show the following important result.

 \begin{lem}\label{asym-lem1}
 Suppose that \eqref{gl-exist-thm-eq2} holds. If $\inf_{x\in\R^N}u_0(x)>0$, then
 \begin{equation}
 \inf_{x\in\R^N}u(x,t;u_0)>0,\ \quad \forall\ t>0.
 \end{equation}
 \end{lem}
 \begin{proof}
 Let $K:=\chi_{1}\lambda_{1}\sup_{x\in\R^N}v_{1}(x,t;u_0)$. Thus, it follows from \eqref{global-exist-eq002} that
 \begin{align*}
u_{t}(\cdot,\cdot;u_0)\geq &\Delta u +\nabla(\chi_{2}v_2(\cdot,\cdot;u_0)-\chi_{1}v_1(\cdot,\cdot;u_0))\nabla u(\cdot,\cdot;u_0)\\
&\,\,\, +(a-K-(b+\chi_2\mu_2-\chi_1\mu_1)u(\cdot,\cdot;u_0))u(\cdot,\cdot;u_0).
\end{align*}
Hence, comparison principle for parabolic equations implies that
$$
u(x,t;u_0)\geq W(t),\quad \forall\ t\geq 0, \ x\in\R^N,
$$
where $W$ is the solution of the ODE
$$
\begin{cases}
W_t=W(a-K-(b+\chi_2\mu_2-\chi_{1}\mu_1)W),\ t>0,\\
W(0)=\inf_{x\in\R^N}u_0(x).
\end{cases}
$$
Since $b+\chi_{2}\mu_2-\chi_1\mu_1>0$, and $\inf_{x\in\R^{N}}u_{0}(x)>0$, we have that $W(t)$ is defined for all time and satisfies $W(t)>0$ for every $t\geq 0$. Hence, we obtain that
$0<W(t)\leq \inf_{x\in\R^N}u(x,t;u_0)$ for all $ t\geq 0$.
 \end{proof}

\begin{proof}[Proof of Theorem B] We divide the proof into two cases.

\smallskip

\noindent {\bf Case I.}  Assume that $b+\chi_2\mu_2-\chi_1\mu_1-\frac{1}{\lambda_2}\Big[|\chi_1\mu_1\lambda_1- \chi_2\mu_2\lambda_2|+\chi_1\mu_1|\lambda_1-\lambda_2| \Big]>0$.

\smallskip

For every $t\geq T_\varepsilon$ {($T_\epsilon$ is such that \eqref{asym-eq03} holds)}, and $x\in\R^N$, we have that
\begin{eqnarray}\label{asym-eq1}
(\chi_2\lambda_2v_2-\chi_1\lambda_1v_1)(x,t;u_0)& = & (\chi_2\mu_2\lambda_2-\chi_1\mu_1\lambda_1)\int_0^\infty\int_{\R^N}e^{-\lambda_2 s}\frac{e^{-\frac{|x-z|^2}{4s}}}{(4\pi s)^{\frac{N}{2}}}u(z,t;u_0)dzds\nonumber\\
& + &\chi_1\mu_1\lambda_1\int_0^\infty\int_{\R^N}(e^{-\lambda_2 s}-e^{-\lambda_1 s})\frac{e^{-\frac{|x-z|^2}{4s}}}{(4\pi s)^{\frac{N}{2}}}u(z,t;u_0)dzdt\nonumber\\
&\leq & \frac{1}{\lambda_2}\left[(\chi_2\mu_2\lambda_2-\chi_1\mu_1\lambda_1)_{+}+\chi_{1}\mu_{1}(\lambda_1-\lambda_2)_{+}\right](\overline{u}+\varepsilon)\nonumber\\
& & -\frac{1}{\lambda_2}\left[(\chi_2\mu_2\lambda_2-\chi_1\mu_1\lambda_1)_{-}+\chi_1\mu_1(\lambda_1-\lambda_2)_{-}\right](\underline{u}-\varepsilon)
\end{eqnarray}
and
\begin{eqnarray}\label{asym-eq2}
(\chi_2\lambda_2v_2-\chi_1\lambda_1v_1)(x,t;u_0)& = & (\chi_2\mu_2\lambda_2-\chi_1\mu_1\lambda_1)\int_0^\infty\int_{\R^N}e^{-\lambda_2 s}\frac{e^{-\frac{|x-z|^2}{4s}}}{(4\pi s)^{\frac{N}{2}}}u(z,t;u_0)dzds\nonumber\\
& + &\chi_1\mu_1\lambda_1\int_0^\infty\int_{\R^N}(e^{-\lambda_2 s}-e^{-\lambda_1 s})\frac{e^{-\frac{|x-z|^2}{4s}}}{(4\pi s)^{\frac{N}{2}}}u(z,t;u_0)dzdt\nonumber\\
&\geq & \frac{1}{\lambda_2}\left[(\chi_2\mu_2\lambda_2-\chi_1\mu_1\lambda_1)_{+}+\chi_{1}\mu_{1}(\lambda_1-\lambda_2)_{+}\right](\underline{u}-\varepsilon)\nonumber\\
& & -\frac{1}{\lambda_2}\left[(\chi_2\mu_2\lambda_2-\chi_1\mu_1\lambda_1)_{-}+\chi_1\mu_1(\lambda_1-\lambda_2)_{-}\right](\overline{u}+\varepsilon).
\end{eqnarray}
Hence,  for every $t\geq T_\varepsilon$, $x\in\R^N$, it follows from \eqref{global-exist-eq002}, \eqref{asym-eq03} and \eqref{asym-eq1} that
\begin{eqnarray}
u_t&\leq & \Delta u+\nabla(\chi_2v_2-\chi_1v_1)\nabla u + (a +\frac{1}{\lambda_2}\left[(\chi_2\mu_2\lambda_2-\chi_1\mu_1\lambda_1)_{+}+\chi_{1}\mu_{1}(\lambda_1-\lambda_2)_{+}\right](\overline{u}+\varepsilon) )u\nonumber\\
& & -(\frac{1}{\lambda_2}\left[(\chi_2\mu_2\lambda_2-\chi_1\mu_1\lambda_1)_{-}+\chi_1\mu_1(\lambda_1-\lambda_2)_{-}\right](\underline{u}-\varepsilon)  +(b+\chi_2\mu_2-\chi_1\mu_1)u)u.
\end{eqnarray}
Thus, by comparison principle for parabolic equations, we have that
\begin{equation}\label{asym-eq3}
u(x,t;u_0)\leq U_{\varepsilon}(t), \quad \forall x\in\R^N,\ t\geq T_\varepsilon,
\end{equation}
where
$U_\varepsilon(t)$ is the solution of the ODE
$$
\begin{cases}
\partial_t U=\Big(a +\frac{1}{\lambda_2}\left[(\chi_2\mu_2\lambda_2-\chi_1\mu_1\lambda_1)_{+}+\chi_{1}\mu_{1}(\lambda_1-\lambda_2)_{+}\right](\overline{u}+\varepsilon) \Big)U\nonumber\\
\ \ \ \ \ -\Big(\frac{1}{\lambda_2}\left[(\chi_2\mu_2\lambda_2-\chi_1\mu_1\lambda_1)_{-}+\chi_1\mu_1(\lambda_1-\lambda_2)_{-}\right](\underline{u}-\varepsilon)  +(b+\chi_2\mu_2-\chi_1\mu_1)U\Big)U \quad t>T_\varepsilon,\cr
U(T_\varepsilon)=\|u(\cdot,T_\varepsilon;u_0)\|_{\infty}.
\end{cases}
$$
Since $b+\chi_2\mu_2-\chi_1\mu_1>0$ and $\|u(\cdot,T_{\infty};u_0)\|_{\infty}>0$, we have that $U_{\varepsilon}(t)$ is defined for all time $t\geq T_\varepsilon$ and satisfies
\begin{align*}
\lim_{t\to\infty}U_{\varepsilon}=&\frac{1}{b+\chi_2\mu_2-\chi_1\mu_1}\Big\{a+\frac{1}{\lambda_2}\left[(\chi_2\mu_2\lambda_2-\chi_1\mu_1\lambda_1)_{+}+\chi_{1}
\mu_{1}(\lambda_1-\lambda_2)_{+}\right](\overline{u}+\varepsilon)\\
 &\qquad\qquad\qquad\qquad  -\frac{1}{\lambda_2}\left[(\chi_2\mu_2\lambda_2-\chi_1\mu_1\lambda_1)_{-}+\chi_1\mu_1(\lambda_1-\lambda_2)_{-}\right](\underline{u}-\varepsilon)\Big\}_+
\end{align*}
This combined with \eqref{asym-eq3} yield that
\begin{align*}
\overline{u}&\leq  \frac{1}{b+\chi_2\mu_2-\chi_1\mu_1}
 \Big\{a+\frac{1}{\lambda_2}\left[(\chi_2\mu_2\lambda_2-\chi_1\mu_1\lambda_1)_{+}+\chi_{1}\mu_{1}(\lambda_1-\lambda_2)_{+}\right](\overline{u}+\varepsilon)
\\
 & \qquad\qquad\qquad\qquad \quad -\frac{1}{\lambda_2}\left[(\chi_2\mu_2\lambda_2-\chi_1\mu_1\lambda_1)_{-}+\chi_1\mu_1(\lambda_1-\lambda_2)_{-}\right](\underline{u}-\varepsilon)\Big\}_+.
\end{align*}
Letting $\varepsilon$ goes to $0$ in the last inequality, we obtain that
\begin{align*}
\overline{u}&\leq  \frac{1}{b+\chi_2\mu_2-\chi_1\mu_1}
 \Big\{a+\frac{1}{\lambda_2}\left[(\chi_2\mu_2\lambda_2-\chi_1\mu_1\lambda_1)_{+}+\chi_{1}\mu_{1}(\lambda_1-\lambda_2)_{+}\right]\overline{u}
\\
 &\qquad\qquad\qquad\quad \qquad -\frac{1}{\lambda_2}\left[(\chi_2\mu_2\lambda_2-\chi_1\mu_1\lambda_1)_{-}+\chi_1\mu_1(\lambda_1-\lambda_2)_{-}\right]\underline{u}\Big\}_+.
\end{align*}
If \begin{align*}
\Big\{&a+\frac{1}{\lambda_2}\left[(\chi_2\mu_2\lambda_2-\chi_1\mu_1\lambda_1)_{+}+\chi_{1}\mu_{1}(\lambda_1-\lambda_2)_{+}\right]\overline{u}\\
&- \frac{1}{\lambda_2}\left[(\chi_2\mu_2\lambda_2-\chi_1\mu_1\lambda_1)_{-}+\chi_1\mu_1(\lambda_1-\lambda_2)_{-}\right]\underline{u}\Big\}_+=0,
 \end{align*}
 then $\overline{u}=\underline{u}=0$. This in turn yields that
\begin{align*}
0=&\Big\{a+\frac{1}{\lambda_2}\left[(\chi_2\mu_2\lambda_2-\chi_1\mu_1\lambda_1)_{+}+\chi_{1}\mu_{1}(\lambda_1-\lambda_2)_{+}\right]\overline{u}\\
&-\frac{1}{\lambda_2}\left[(\chi_2\mu_2\lambda_2-\chi_1\mu_1\lambda_1)_{-}+\chi_1\mu_1(\lambda_1-\lambda_2)_{-}\right]\underline{u}\Big\}_+=a,
\end{align*}
which is impossible, since $a>0$. Hence
\begin{align}\label{asym-eq4}
\overline{u}\leq &\frac{1}{b+\chi_2\mu_2-\chi_1\mu_1}\Big[ a+\frac{1}{\lambda_2}\left[(\chi_2\mu_2\lambda_2-\chi_1\mu_1\lambda_1)_{+}+\chi_{1}\mu_{1}(\lambda_1-\lambda_2)_{+}\right]\overline{u}\nonumber\\
&\qquad\qquad\qquad \qquad  -\frac{1}
{\lambda_2}\left[(\chi_2\mu_2\lambda_2-\chi_1\mu_1\lambda_1)_{-}+\chi_1\mu_1(\lambda_1-\lambda_2)_{-}\right]\underline{u}\Big].
\end{align}

On the other hand, for every $t\geq T_\varepsilon$, $x\in\R^N$, it follows from \eqref{global-exist-eq002}, \eqref{asym-eq03} and \eqref{asym-eq1} that
\begin{eqnarray}
u_t&\geq & \Delta u+\nabla(\chi_2v_2-\chi_1v_1)\nabla u + \Big(a +\frac{1}{\lambda_2}\left[(\chi_2\mu_2\lambda_2-\chi_1\mu_1\lambda_1)_{+}+\chi_{1}\mu_{1}(\lambda_1-\lambda_2)_{+}\right](\underline{u}-\varepsilon) \Big)u\nonumber\\
& & -\Big(\frac{1}{\lambda_2}\left[(\chi_2\mu_2\lambda_2-\chi_1\mu_1\lambda_1)_{-}+\chi_1\mu_1(\lambda_1-\lambda_2)_{-}\right](\overline{u}+\varepsilon)  +(b+\chi_2\mu_2-\chi_1\mu_1)u\Big)u.
\end{eqnarray}
Thus, by comparison principle for parabolic equations, we have that
\begin{equation}\label{asym-eq5}
u(x,t;u_0)\geq U^{\varepsilon}(t), \quad \forall x\in\R^N,\ t\geq T_\varepsilon,
\end{equation}
where $U^\varepsilon(t)$ is the solution of the ODE
$$
\begin{cases}
\partial_t U=\Big(a +\frac{1}{\lambda_2}\left[(\chi_2\mu_2\lambda_2-\chi_1\mu_1\lambda_1)_{+}+\chi_{1}\mu_{1}(\lambda_1-\lambda_2)_{+}\right](\underline{u}-\varepsilon) \Big)U\nonumber\\
\ \ \ \ \ -\Big(\frac{1}{\lambda_2}\left[(\chi_2\mu_2\lambda_2-\chi_1\mu_1\lambda_1)_{-}+\chi_1\mu_1(\lambda_1-\lambda_2)_{-}\right](\overline{u}+\varepsilon)  +(b+\chi_2\mu_2-\chi_1\mu_1)U\Big)U, \quad t>T_\varepsilon\cr
U(T_\varepsilon)=\inf_{x\in\R^N}u(x,T_\varepsilon).
\end{cases}
$$

But, by Lemma \ref{asym-lem1} we have that $\inf_{x\in \R^N}u(x,T_\varepsilon;u_0) >0$. Since $b+\chi_2\mu_2-\chi_1\mu_1>0$, we have that $U^{\varepsilon}(t)$ is defined for all time $t\geq T_\varepsilon$ and satisfies
\begin{align*}
\lim_{t\to\infty}U^{\varepsilon}&= \frac{1}{b+\chi_2\mu_2-\chi_1\mu_1}\Big\{a +\frac{1}{\lambda_2}\left[(\chi_2\mu_2\lambda_2-\chi_1\mu_1\lambda_1)_{+}+\chi_{1}\mu_{1}(\lambda_1-\lambda_2)_{+}\right](\underline{u}-\varepsilon)
\\
&\quad \qquad\qquad\qquad\qquad  -(\frac{1}{\lambda_2}\left[\chi_2\mu_2\lambda_2-\chi_1\mu_1\lambda_1)_{-}+\chi_1\mu_1(\lambda_1-\lambda_2)_{-}\right](\overline{u}+\varepsilon)\Big\}_{+}.
\end{align*}
This combined with \eqref{asym-eq5} yield that
\begin{align*}
\underline{u}&\geq\frac{1}{b+\chi_2\mu_2-\chi_1\mu_1}
\Big\{a +\frac{1}{\lambda_2}\left[(\chi_2\mu_2\lambda_2-\chi_1\mu_1\lambda_1)_{+}+\chi_{1}\mu_{1}(\lambda_1-\lambda_2)_{+}\right](\underline{u}-\varepsilon)\\
& \qquad\qquad\qquad \qquad\qquad -(\frac{1}{\lambda_2}\left[\chi_2\mu_2\lambda_2-\chi_1\mu_1\lambda_1)_{-}+\chi_1\mu_1(\lambda_1-\lambda_2)_{-}\right](\overline{u}+\varepsilon)\Big\}_{+}.
\end{align*}
Letting $\varepsilon$ goes to $0$ in the last inequality, we obtain that
\begin{align}
\label{asym-eq6}
\underline{u}&\geq\frac{1}{b+\chi_2\mu_2-\chi_1\mu_1}
\Big\{a +\frac{1}{\lambda_2}\left[(\chi_2\mu_2\lambda_2-\chi_1\mu_1\lambda_1)_{+}+\chi_{1}\mu_{1}(\lambda_1-\lambda_2)_{+}\right]\underline{u}\nonumber\\
& \quad -(\frac{1}{\lambda_2}\left[\chi_2\mu_2\lambda_2-\chi_1\mu_1\lambda_1)_{-}+\chi_1\mu_1(\lambda_1-\lambda_2)_{-}\right]\overline{u}\Big\}_{+}.
\end{align}
%\begin{equation}\label{asym-eq6}
%\underline{u}\geq\frac{a +\frac{1}{\lambda_2}\left[(\chi_2\mu_2\lambda_2-\chi_1\mu_1\lambda_1)_{+}+\chi_{1}\mu_{1}(\lambda_1-\lambda_2)_{+}\right]\underline{u}-(\frac{1}{\lambda_2}\left[\chi_2\mu_2\lambda_2-\chi_1\mu_1\lambda_1)_{-}+\chi_1\mu_1(\lambda_1-\lambda_2)_{-}\right]\overline{u}}{b+\chi_2\mu_2-\chi_1\mu_1}.
%\end{equation}
It follows from inequalities \eqref{asym-eq4} and \eqref{asym-eq6} that
\begin{equation}\label{asym-eq7}
\Big(b+\chi_2\mu_2-\chi_1\mu_1-\frac{1}{\lambda_2}\big[|\chi_1\mu_1\lambda_1- \chi_2\mu_2\lambda_2|+\chi_1\mu_1|\lambda_1-\lambda_2| \big]\Big)(\overline{u}-\underline{u})\leq 0.
\end{equation}

%That is
%$$ \left(b+\chi_2\mu_2-\chi_1\mu_1-\frac{1}{\lambda_2}\left( (\chi_2\mu_2\lambda_2-\chi_1\mu_1\lambda_1)+\chi_1\mu_1(\lambda_2-\lambda_1)\right)\right)(\overline{u}-\underline{u})\leq 0$$
%Observe that $$ b+\chi_2\mu_2-\chi_1\mu_1-\frac{1}{\lambda_2}\left( (\chi_2\mu_2\lambda_2-\chi_1\mu_1\lambda_1)+\chi_1\mu_1(\lambda_2-\lambda_1)\right)=b - \frac{2}{\lambda_2}(\lambda_2-\lambda_1)\chi_1\mu_1.$$
 In this case, it follows from inequality \eqref{asym-eq7} that $\overline{u}=\underline{u}$.   Combining this with \eqref{asym-eq4} and \eqref{asym-eq6}, we obtain that $\overline{u}=\underline{u}=\frac{a}{b}$. This end the first case.

 \medskip

\noindent {\bf Case II.}  Assume that $b+\chi_2 \mu_2 -\chi_1\mu_1 -\frac{1}{\lambda_{1}}\big[|\chi_1\mu_1\lambda_1-\chi_2\mu_2\lambda_2|+\chi_2\mu_2|\lambda_1-\lambda_2|\big]>0$.

\smallskip

Rewrite $\chi_2\lambda_2v_2-\chi_1\lambda_1v_1$ in the form
\begin{eqnarray}\label{asym-eq8_1}
(\chi_2 \lambda_2 v_2-\chi_1 \lambda_1 v_1)(x,t;u_0)& = & \chi_2\mu_2\lambda_2\int_0^\infty\int_{\R^N}(e^{-\lambda_2 s}-e^{-\lambda_{1}s})\frac{-\frac{|x-z|^2}{4s}}{(4\pi s)^{\frac{N}{2}}}u(z,t;u_0)dzds\nonumber\\
& &+ (\chi_2\mu_2\lambda_2-\chi_1\mu_1\lambda_1)\int_0^\infty\int_{\R^N}e^{-\lambda_1 s}\frac{e^{-\frac{|x-z|^2}{4s}}}{(4\pi s)^{\frac{N}{2}}}u(z,t;u_0)dzds.\nonumber\\
\end{eqnarray}
It follows from the arguments used to establish inequalities  \eqref{asym-eq4} and \eqref{asym-eq6} that
\begin{align}\label{asym-eq8}
\overline{u}\,\, \leq\,\,  \frac{1}{ b+\chi_2\mu_2-\chi_1\mu_1} \Big\{&a+\frac{1}{\lambda_1}\left[(\chi_2\mu_2(\lambda_1-\lambda_2)_{+}+\chi_2\mu_2(\lambda_1-\lambda_2)_{+}\right]\overline{u} \nonumber\\
& -\frac{1}{\lambda_1}\left[(\chi_2\mu_2\lambda_2-\chi_1\mu_1\lambda_1)_{-}+\chi_2\mu_2(\lambda_1-\lambda_2)_+\right]\underline{u}\Big\}
\end{align}
and
\begin{align}\label{asym-eq9}
\underline{u}\geq \frac{1}{ b+\chi_2\mu_2-\chi_1\mu_1}
\Big\{&a+\frac{1}{\lambda_1}\left[(\chi_2\mu_2(\lambda_1-\lambda_2)_{+}+\chi_2\mu_2(\lambda_1-\lambda_2)_{+}\right]\underline{u}\nonumber\\ &-\frac{1}{\lambda_1}\left[(\chi_2\mu_2\lambda_2-\chi_1\mu_1\lambda_1)_{-}+\chi_2\mu_2(\lambda_1-\lambda_2)_+\right]\overline{u}\Big\}
\end{align}
hold respectively.
It follows from \eqref{asym-eq8} and \eqref{asym-eq9} that
\begin{equation}\label{asym-eq10}
\Big( b+\chi_2 \mu_2 -\chi_1\mu_1 -\frac{1}{\lambda_{1}}\big[|\chi_1\mu_1\lambda_1-\chi_2\mu_2\lambda_2|+\chi_2\mu_2|\lambda_1-\lambda_2|\big]\Big)(\overline{u}-\underline{u})\leq 0.
\end{equation}
Since $b+\chi_2 \mu_2 -\chi_1\mu_1 -\frac{1}{\lambda_{1}}\left[|\chi_1\mu_1\lambda_1-\chi_2\mu_2\lambda_2|+\chi_2\mu_2|\lambda_1-\lambda_2|\right]>0$, it follows from inequality \eqref{asym-eq10} that $\overline{u}=\underline{u}$.   Combining this with \eqref{asym-eq8} and \eqref{asym-eq9}, we obtain that $\overline{u}=\underline{u}=\frac{a}{b}$. This end the second case.

Therefore, it follows from the results of cases I and II that if
 \begin{align*}
 b+\chi_2\mu_2-\chi_1\mu_1 > \min\Big \{&\frac{1}{\lambda_{2}}\left[|\chi_1\mu_1\lambda_1-\chi_2\mu_2\lambda_2|+\chi_1\mu_1|\lambda_1-\lambda_2|\right],\\ &\frac{1}{\lambda_{1}}\left[|\chi_1\mu_1\lambda_1-\chi_2\mu_2\lambda_2|+\chi_2\mu_2|\lambda_1-\lambda_2|\right]\Big\},
 \end{align*}
then $\overline{u}=\underline{u}=\frac{a}{b}$. Thus Theorem B follows.
\end{proof}

\section{Spreading properties of classical solutions }

In this section we study how fast the mobiles species spread over time and prove Theorems C and  D. Throughout this section, we always suppose that $u_0\in C^{b}_{\rm unif}(\R^N)$, $u_0(x)\geq 0$  has compact and nonempty support.  The next three lemmas will be useful in the subsequent.

\begin{lem}\label{asym-lem2}
Let $u_0\in C_{\rm unif}^{b}(\R^N)$, $u_0\geq 0$, and $(u(\cdot,\cdot;u_0),v_{1}(\cdot,\cdot;u_0),v_2(\cdot,\cdot;u_0))$ be the classical solution of \eqref{Main Intro-eq} with $u(\cdot,0;u_0)=u_0$. Then for every $i\in\{1,\cdots,N\}$, we have that
\begin{align}\label{asym-eq16}
&\|\partial_{x_i}(\chi_2v_2-\chi_1v_1)(\cdot,t;u_0)\|_{\infty}\nonumber \\
& \leq \min\Big\{\frac{|\chi_2\mu_2-\chi_1\mu_1|}{2\sqrt{\lambda_2}}+\frac{\chi_1\mu_1|\sqrt{\lambda_1}-\sqrt{\lambda_2}|}{2\sqrt{\lambda_1\lambda_2}}, \frac{|\chi_1\mu_1-\chi_2\mu_2|}{2\sqrt{\lambda_1}}+\frac{\chi_2\mu_2|\sqrt{\lambda_2}-\sqrt{\lambda_1}|}{2\sqrt{\lambda_1\lambda_2}}
\Big\}\|u(\cdot,t;u_0)\|_{\infty}
\end{align}
for every   $t\geq 0$.
\end{lem}
\begin{proof}For every $i\in\{1,\cdots,N\}$ and $k\in\{1,2\}$, we have
 that
\begin{align*}
& \partial_{x_i}(\chi_kv_k)(x,t;u_0)\nonumber\\
 & =\frac{\chi_k\mu_k}{\pi^{\frac{N}{2}}}\int_0^\infty\int_{\R^N}\frac{ e^{-\lambda_k s}}{\sqrt{ s}}z_i e^{-|z|^2}u(x+2\sqrt{s}z,t;u_0)dzds  \nonumber\\
&=\frac{\chi_k\mu_k}{\pi^{\frac{N}{2}}}\int_0^\infty\int_{\R^{N-1}}\frac{ e^{-\lambda_k s}e^{-|y|^2}}{\sqrt{ s}}\left[\int_{\R}\tau e^{-\tau^2}u(x+2\sqrt{s}\tau e_i +2\sqrt{s} \pi_i^{-1}(y),t;u_0)d\tau\right]dyds
\nonumber\\
&=\frac{\chi_k\mu_k}{\pi^{\frac{N}{2}}}\int_0^\infty\int_{\R^{N-1}}\frac{ e^{-\lambda_k s}e^{-|y|^2}}{\sqrt{ s}}\left[\int_0^\infty \tau e^{-\tau^2}u(x+2\sqrt{s}\tau e_i +2\sqrt{s} \pi_i^{-1}(y),t;u_0)d\tau\right]dyds
\nonumber\\
& - \frac{\chi_k\mu_k}{\pi^{\frac{N}{2}}}\int_0^\infty\int_{\R^{N-1}}\frac{ e^{-\lambda_k s}e^{-|y|^2}}{\sqrt{ s}}\left[\int_0^\infty \tau e^{-\tau^2}u(x-2\sqrt{s}\tau e_i +2\sqrt{s} \pi_i^{-1}(y),t;u_0)d\tau\right]dyds,
\end{align*}
 where $e_i=(\delta_{1i},\delta_{2i},\cdots,\delta_{Ni})$ with $\delta_{ij}=0$ if $i\not =j$ and $\delta_{ii}=1$ for
 $i,j=1,2,\cdots,N$, and $\pi_i^{-1}(y)=(y_1,y_2,\cdots,y_{i-1},0,y_i,\cdots,y_{N-1})$.
Hence,
\begin{align}\label{asym-eq15}
& \partial_{x_i}(\chi_2v_2-\chi_1v_1)(x,t;u_0)\nonumber\\
&= \frac{(\chi_2\mu_2-\chi_1\mu_1)}{\pi^{\frac{N}{2}}}\int_0^\infty\int_{\R^{N-1}}\frac{e^{-\lambda_2 s}e^{-|y|^2}}{\sqrt{s}}\left[ \int_0^\infty \tau e^{-\tau^2}u(x+2\sqrt{s}\tau e_i +2\sqrt{s}\pi_i^{-1}(y),t;u_0)d\tau \right]dyds\nonumber\\
& + \frac{\chi_1\mu_1}{\pi^{\frac{N}{2}}}\int_0^\infty\int_{\R^{N-1}}\frac{(e^{-\lambda_2 s}-e^{-\lambda_1 s})e^{-|y|^2}}{\sqrt{s}}\left[\int_0^\infty \tau e^{-\tau^2}u(x+2\sqrt{s}\tau e_i +2\sqrt{s} \pi_i^{-1}(y),t;u_0)d\tau\right]dyds\nonumber\\
&+ \frac{(\chi_1\mu_1-\chi_2\mu_2)}{\pi^{\frac{N}{2}}}\int_0^\infty\int_{\R^{N-1}}\frac{e^{-\lambda_2 s}e^{-|y|^2}}{\sqrt{s}}\left[ \int_0^\infty \tau e^{-\tau^2}u(x-2\sqrt{s}\tau e_i +2\sqrt{s} \pi_i^{-1}(y),t;u_0)d\tau \right]dyds\nonumber\\
& + \frac{\chi_1\mu_1}{\pi^{\frac{N}{2}}}\int_0^\infty\int_{\R^{N-1}}\frac{(e^{-\lambda_1 s}-e^{-\lambda_2 s})e^{-|y|^2}}{\sqrt{s}}\left[\int_0^\infty \tau e^{-\tau^2}u(x-2\sqrt{s}\tau e_i +2\sqrt{s} \pi_i^{-1}(y),t;u_0)d\tau\right]dyds.
\end{align}
Using the fact that $\int_0^\infty\frac{e^{-\lambda_k s}}{\sqrt{s}}ds=\frac{\sqrt{\pi}}{\sqrt{\lambda_k}}$, $\int_0^{\infty}\tau e^{-\tau^2}d\tau=\frac{1}{2}$, $\int_{\R^{N-1}}e^{-|y|^2}dy=\pi^{\frac{N-1}{2}}$, it follows from \eqref{asym-eq15} that for every $x\in\R^N,\ t\geq 0$, we have
\begin{align*}
|\partial_{x_i}(\chi_2v_2-\chi_1v_1)(x,t;u_0)| \leq \left[ \frac{|\chi_2\mu_2-\chi_1\mu_1|}{2\sqrt{\lambda_2}}+\frac{\chi_1\mu_1|\sqrt{\lambda_1}-\sqrt{\lambda_2}|}{2\sqrt{\lambda_1\lambda_2}}\right]\|u(\cdot,t;u_0)\|_{\infty}.
\end{align*}
Similarly, we have that
\begin{align*}
|\partial_{x_i}(\chi_1v_1-\chi_2v_2)(x,t;u_0)| \leq \left[ \frac{|\chi_1\mu_1-\chi_2\mu_2|}{2\sqrt{\lambda_1}}+\frac{\chi_2\mu_2|\sqrt{\lambda_2}-\sqrt{\lambda_1}|}{2\sqrt{\lambda_1\lambda_2}}\right]\|u(\cdot,t;u_0)\|_{\infty}.
\end{align*}
The lemma thus follows.
\end{proof}

\begin{lem}\label{Asym-lem3}
Suppose that \eqref{gl-exist-thm-eq2} holds. Let $u_0\in C_{\rm unif}^{b}(\R^N)$, $u_0\geq 0$, and $(u(\cdot,\cdot;u_0),v_{1}(\cdot,\cdot;u_0)$, $v_2(\cdot,\cdot;u_0))$ be the classical solution of \eqref{Main Intro-eq} with $u(\cdot,0;u_0)=u_0$. Then we have that
\begin{equation}\label{asym-eq17}
\limsup_{t\to\infty}\|u(\cdot,t;u_0)\|_{\infty}\leq \frac{a}{b+\chi_2\mu_2-\chi_1\mu_1-M},
\end{equation}
where $M$ is given by \eqref{gl-exist-thm-eq1}.
\end{lem}
\begin{proof}
It follows from inequalities \eqref{asym-eq4} and \eqref{asym-eq8} that
$$
\overline{u}\leq \frac{a+\frac{1}{\lambda_2}\left[((\chi_2\mu_2\lambda_2-\chi_1\mu_1\lambda_1)_{+}+\chi_1\mu_1(\lambda_1-\lambda_2)_{+} ) \right]\overline{u}}{b+\chi_2\mu_2-\chi_1\mu_1}$$
and
$$\overline{u}\leq \frac{a+\frac{1}{\lambda_1}\left[((\chi_2\mu_2\lambda_2-\chi_1\mu_1\lambda_1)_{+}+\chi_2\mu_2(\lambda_1-\lambda_2)_{+} ) \right]\overline{u}}{b+\chi_2\mu_2-\chi_1\mu_1}.
$$ Which is equivalent to
$$
(b+\chi_2\mu_2-\chi_1\mu_1)\overline{u}\leq a+\frac{1}{\lambda_2}\left[((\chi_2\mu_2\lambda_2-\chi_1\mu_1\lambda_1)_{+}+\chi_1\mu_1(\lambda_1-\lambda_2)_{+} ) \right]\overline{u}$$
and
$$(b+\chi_2\mu_2-\chi_1\mu_1)\overline{u}\leq a+\frac{1}{\lambda_1}\left[((\chi_2\mu_2\lambda_2-\chi_1\mu_1\lambda_1)_{+}+\chi_2\mu_2(\lambda_1-\lambda_2)_{+} ) \right]\overline{u}.
$$
Hence
$$(b+\chi_2\mu_2-\chi_1\mu_1)\overline{u}\leq a+M\overline{u}.$$
The lemma thus follows.
\end{proof}

\medskip

\begin{lem}\label{asym-lem4}
\begin{description}
\item[1)] If there is a positive constant $c^*_{-}(\chi_1,\mu_1,\lambda_2,\chi_2,\mu_2,\lambda_2)$ such that
\begin{equation}
\lim_{t\to\infty}\sup_{|x|\leq ct}|u(x,t;u_0)-\frac{a}{b}|=0\quad \forall\ 0\leq c< c^*_{-}(\chi_1,\mu_1,\lambda_2,\chi_2,\mu_2,\lambda_2),
\end{equation}
then for every $i=1,2$ we have
\begin{equation}
\lim_{t\to\infty}\sup_{|x|\leq ct}|\lambda_i v_i(x,t;u_0)-\frac{a}{b}\mu_i|=0\quad \forall\ 0\leq c< c^*_{-}(\chi_1,\mu_1,\lambda_2,\chi_2,\mu_2,\lambda_2).
\end{equation}
\item[2)] If there is a positive constant $ c^*_{+}(\chi_1,\mu_1,\lambda_2,\chi_2,\mu_2,\lambda_2)$ such that
\begin{equation}
\lim_{t\to\infty}\sup_{|x|\geq ct}u(x,t;u_0)=0 \quad \forall\ c> c^*_{+}(\chi_1,\mu_1,\lambda_2,\chi_2,\mu_2,\lambda_2),
\end{equation}
then for each $i=1,2$ we have that
\begin{equation}
\lim_{t\to\infty}\sup_{|x|\geq ct}v_i(x,t;u_0)=0 \quad \forall\ c> c^*_{+}(\chi_1,\mu_1,\lambda_2,\chi_2,\mu_2,\lambda_2).
\end{equation}
\end{description}
\end{lem}
 The proof of Lemma \ref{asym-lem4} follows from the proof of Lemma 5.5 \cite{SaSh1}.

Now, we are ready to prove Theorem C.

\begin{proof}[Proof of Theorem C]
%Consider $$\mathcal{E}^{T}:=\{ u\in C^{b}_{\rm unif}(\R^N\times [0\ , \ T])\ |\ u(\cdot,0)=u_0, 0\leq u(x,t)\leq \varphi(x,t), \ x\in\R^N, 0\leq t\leq T \}. $$

%{\bf Case (i)} Suppose $\lambda_1\leq \lambda_2$, $\chi_2\mu_2\lambda_2\geq \chi_1\mu_1\lambda_1$, and $b-(1-\frac{\lambda_1}{\lambda_2})\chi_2\mu_2>0$.\\ It follows from \eqref{asym-eq4} that  \begin{eqnarray*}(b+\chi_2\mu_2-\chi_1\mu_1)\overline{u} & \leq &  a+\frac{1}{\lambda_2}(\chi_2\mu_2\lambda_2-\chi_1\mu_1\lambda_1)\overline{u}+\frac{1}{\lambda_2}\chi_1\mu_1(\lambda_1-\lambda_2)\underline{u} \nonumber\\  & \leq & a + \frac{1}{\lambda_2}(\chi_2\mu_2\lambda_2-\chi_1\mu_1\lambda_1)\overline{u}. \end{eqnarray*} From what it follows that  \begin{equation}\label{asym-eq17} \overline{u}\leq \frac{a}{b-(1-\frac{\lambda_1}{\lambda_2})\chi_1\mu_1}.\end{equation} Note in this case that $b-(1-\frac{\lambda_1}{\lambda_2})\chi_1\mu_1=b+\chi_2\mu_2-\chi_1\mu_1-M$, where $M$ is given by \eqref{gl-exist-thm-eq1}.
Combining inequalities \eqref{asym-eq16} and \eqref{asym-eq17}, we obtain that
\begin{equation}\label{asym-eq18}
\|\nabla(\chi_2v_2-\chi_1v_1)(\cdot,t;u_0)\|_{\infty}\leq \frac{aD\sqrt{N}}{b+\chi_2\mu_2-\chi_1\mu_1-M} +D\varepsilon\sqrt{N}, \quad \forall\ t\geq T_\varepsilon
\end{equation}
where $D$ is given by \eqref{asym-hyp1} and $M$ given by \eqref{gl-exist-thm-eq1}. Let
$$ K_{\varepsilon}:= \sup_{0\leq t\leq T_\varepsilon}\|\nabla (\chi_2 v_2-\chi_1 v_1)(\cdot,t;u_0)\|_{\infty} \quad \text{and}\quad K_{\varepsilon,2}=:\sup_{0\leq t\leq T_\varepsilon}\|\frac{\chi_2\lambda_2}{\sqrt{a	}}v_{2}(\cdot,t;u_0)\|_{\infty}.
$$
Choose $C>0$ such that $$u_{0}(x)\leq C e^{-\sqrt{a}|x|},\quad \forall \ x\in\R^N.$$
Let $\xi\in S^{N-1}$ be given and  consider
%\begin{equation}u_t=\Delta u - K_{\varepsilon}\xi\nabla u+u(a + \overline{M}_\varepsilon-(b+\chi_2\mu_2-\chi_1\mu_1)u) \end{equation}
%Let
$$ \overline{U}(x,t;\xi)= Ce^{-\sqrt{a}(x\cdot\xi -( 2\sqrt{a} +K_\varepsilon + K_{\varepsilon,2})t)}.$$
We have that
\begin{align}
&\overline{U}_t-\Delta \overline{U}-\nabla((\chi_2v_2-\chi_1v_1)(\cdot,\cdot;u_0))\nabla\overline{U}-(a+(\chi_2\lambda_2v_2-\chi_1\lambda_1v_1)(\cdot,\cdot;u_0)-(b+\chi_2\mu_2-\chi_1\mu_1)\overline{U})\overline{U}\nonumber\\
& =\Big(\sqrt{a}(2\sqrt{a}+K_{\varepsilon}+K_{\varepsilon,2})-a+\sqrt{a}\nabla((\chi_2v_2-\chi_1v_1)(\cdot,\cdot;u_0))\cdot\xi\Big)\overline{U}\nonumber\\
 &\qquad - \Big(a +(\chi_2\lambda_2v_2-\chi_1\lambda_1v_1)(\cdot,\cdot;u_0)-(b+\chi_2\mu_2-\chi_1\mu_1)\overline{U} \Big)\overline{U}\nonumber\\
& = \Big( \sqrt{a}\big(K_\varepsilon+\nabla((\chi_2v_2-\chi_1v_1)(\cdot,\cdot;u_0))\cdot\xi\big) +(\sqrt{a}K_{\varepsilon,2}-\chi_2\lambda_2v_2(\cdot,\cdot;u_0))\Big)\overline{U}\nonumber\\
&\ \ +\Big(\chi_1\lambda_1v_1(\cdot,\cdot;u_0) +(b+\chi_2\mu_2-\chi_1\mu_1)\overline{U}\Big)\overline{U}\nonumber\\
& \geq 0 \quad \forall x\in\R^N, \ 0<t\leq T_{\varepsilon},\ \forall\ \xi\in S^{N-1}.
\end{align}
Since $\overline{U}(x,0;\xi)=Ce^{-\sqrt{a}x\cdot\xi}\geq Ce^{-\sqrt{a}|x|}\geq u_0(x)$, by comparison principle for parabolic equations, we obtain that
\begin{equation}
u(x,t;u_0)\leq \overline{U}(x,t;\xi),\quad \forall\ x\in\R^N,\ 0\leq t\leq T_\varepsilon \ \forall\ \xi\in S^{N-1}.
\end{equation}
Next, consider
\begin{equation}
\overline{W}(x,t;\xi)=Ce^{-\sqrt{a}(x\cdot\xi-(2\sqrt{a}+ L_\varepsilon+L_{\varepsilon,2})(t-T_\varepsilon))}e^{\sqrt{a}(2\sqrt{a}+K_\varepsilon+K_{\varepsilon,2})T_{\varepsilon}}, \quad x\in\R^N, \ t\geq T_{\varepsilon},
\end{equation}
where
$$
L_{\varepsilon}:=\frac{aD\sqrt{N}}{b+\chi_2\mu_2-\chi_1\mu_1-M}+D\varepsilon\sqrt{N}$$
and
 $$ L_{\varepsilon,2}:=\frac{\sqrt{a}\chi_2\mu_2}{b+\chi_2\mu_2-\chi_1\mu_1-M} +\frac{
 \chi_2\mu_2\varepsilon}{\sqrt{a}}.
$$
It follows from \eqref{asym-eq03}, \eqref{asym-eq17} and \eqref{asym-eq18} that for any $x\in\R^N$ and $t\ge T_\epsilon$,
$$
\overline{W}_t-\Delta \overline{W}-\nabla((\chi_2v_2-\chi_1v_1)(\cdot,\cdot;u_0))\nabla\overline{W}-(a+(\chi_2\lambda_2v_2-\chi_1\lambda_1v_1)(\cdot,\cdot;u_0)-(b+\chi_2\mu_2-\chi_1\mu_1)\overline{W})\overline{W}\geq 0.
$$
Observe that $\overline{W}(\cdot,T_\varepsilon;\xi)=\overline{U}(\cdot,T_\varepsilon;\xi)\geq u(\cdot,T_\varepsilon)$. Hence, comparison principle for parabolic equations implies that
\begin{equation}
0\leq u(x,t;u_0)\leq \overline{W}(x,t;\xi),\quad x\in\R^N,\ t\geq T_{\varepsilon},\ \xi\in S^{N-1}.
\end{equation}
Hence, for every $c>2\sqrt{a}+L_{\varepsilon}+L_{\varepsilon,2}$, and $t>T_\varepsilon$,  we have
$$
\sup_{|x|\geq ct}u(x,t;u_0)\leq \sup_{|x|\geq ct }\overline{W}(x,t,\frac{1}{|x|}x)\leq \sup_{|x|\geq ct}Me^{-\sqrt{a}(c-(2\sqrt{a}+ L_\varepsilon+L_{\varepsilon,2})(t-T_\varepsilon))}e^{\sqrt{a}(2\sqrt{a}+K_\varepsilon+K_{\varepsilon,2})T_{\varepsilon}} \to 0
$$
as $ t\to \infty$.
Thus by taking
\begin{equation}\label{asym-eq19}
c_{+}^{*}(\chi_1,\mu_1,\lambda_1,\chi_2,\mu_2,\lambda_2):=2\sqrt{a}+\lim_{\varepsilon\to 0^+}(L_{\varepsilon}+L_{\varepsilon,2})=2\sqrt{a}+\frac{\sqrt{a}(D\sqrt{N a}+\chi_2\mu_2)}{b+\chi_2\mu_2-\chi_1\mu_1-M},
\end{equation}
and using Lemma \ref{asym-lem4}, the result of Theorem C follows.
\end{proof}

In order to prove Theorem D, we first establish the following  important Lemma.

\begin{lem}\label{asym-lem5}
Let $L$ be given by \eqref{L-eq}.  Then,
\begin{align}\label{asym-eq20}
&\lim_{ R\to\infty} \inf_{ |x| \geq R,\ T\geq R}\left( 4(a+(\chi_2\lambda_2v_2-\chi_1\lambda_1v_1)(x,t;u_0)-|\nabla(\chi_2v_2-\chi_1v_1)(x,t;u_0)|^2\right)\nonumber\\
&\geq 4a(1-L)-\frac{Na^2D^2}{(b+\chi_2\mu_2-\chi_1\mu_1)^2} .
\end{align}
\end{lem}

\begin{proof}Using inequalities \eqref{asym-eq2} and \eqref{asym-eq16}, we have that for every $t\geq T_{\varepsilon}, \ x\in\R^N$,
\begin{align}\label{asym-eq21}
&4\big(a+(\chi_2\lambda_2v_2-\chi_1\lambda_1v_1)(x,t;u_0)\big) -|\nabla(\chi_2v_2(x,t;u_0)-\chi_1v_1(x,t;u_0))|^2\nonumber\\
&\geq 4\big(a+ \frac{1}{\lambda_2}\left[(\chi_2\mu_2\lambda_2-\chi_1\mu_1\lambda_1)_{+}+\chi_1\mu_1(\lambda_1-\lambda_2)_{+}\right](\underline{u}-\varepsilon)\big)\nonumber\\
&-\frac{4}{\lambda_2}\left[(\chi_2\mu_2\lambda_2-\chi_1\mu_1\lambda_1)_{-}+\chi_1\mu_1(\lambda_1-\lambda_2)_{-}\right](\overline{u}+\varepsilon)-ND^{2}(\overline{u}+\varepsilon)^2.
\end{align}
Letting first $R\to \infty$ and then $\varepsilon\to 0$, it follows from \eqref{asym-eq21} that
\begin{align}\label{asym-eq22}
&\lim_{ R\to\infty} \inf_{ |x| \geq R,\ T\geq R}\left( 4(a+(\chi_2\lambda_2v_2-\chi_1v_1\lambda_1)(x,t;u_0)-|\nabla(\chi_2v_2-\chi_1v_1)(x,t;u_0)|^2\right)\nonumber \\
&\geq 4\big(a+ \frac{1}{\lambda_2}\left[(\chi_2\mu_2\lambda_2-\chi_1\mu_1\lambda_1)_{+}+\chi_1\mu_1(\lambda_1-\lambda_2)_{+}\right]\underline{u}\big)\nonumber\\
&-\frac{4}{\lambda_2}\left[(\chi_2\mu_2\lambda_2-\chi_1\mu_1\lambda_1)_{-}+\chi_1\mu_1(\lambda_1-\lambda_2)_{-}\right]\overline{u}-ND^{2}\overline{u}^2.
\end{align}
But Theorem C implies  that $\underline{u}=0$. Hence, inequality \eqref{asym-eq22} implies that
\begin{align}\label{asym-eq23}
& \lim_{ R\to\infty} \inf_{ |x| \geq R,\ T\geq R}\left( 4(a+(\chi_2\lambda_2v_2-\chi_1v_1\lambda_1)(x,t;u_0)-|\nabla(\chi_2v_2-\chi_1v_1)(x,t;u_0)|^2\right)\nonumber \\
&\geq 4\big(a-\frac{1}{\lambda_2}\left[(\chi_2\mu_2\lambda_2-\chi_1\mu_1\lambda_1)_{-}+\chi_1\mu_1(\lambda_1-\lambda_2)_{-}\right]\overline{u}\big)-ND^{2}\overline{u}^2.
\end{align}
Thus, it follows from  \eqref{asym-eq23} and \eqref{asym-eq17} that
\begin{align}\label{asym-eq24}
& \lim_{ R\to\infty} \inf_{ |x| \geq R,\ T\geq R}\left( 4(a+(\chi_2\lambda_2v_2-\chi_1v_1\lambda_1)(x,t;u_0)-|\nabla(\chi_2v_2-\chi_1v_1)(x,t;u_0)|^2\right)\nonumber \\
&\geq 4\big(a-\frac{a\left[(\chi_2\mu_2\lambda_2-\chi_1\mu_1\lambda_1)_{-}+\chi_1\mu_1(\lambda_1-\lambda_2)_{-}\right]}{\lambda_2(b+\chi_2\mu_2-\chi_1\mu_1-M)}\big)-\frac{ND^{2}a^2}{(b+\chi_2\mu_2-\chi_1\mu_1-M)^2}.
\end{align}
Similarly, by rewriting $(\chi_2\mu_2v_2-\chi_1\mu_1v_1)(x,t;u_0)$ in the form given by \eqref{asym-eq8_1}, same arguments as above yield that
\begin{align}\label{asym-eq25}
& \lim_{ R\to\infty} \inf_{ |x| \geq R,\ T\geq R}\left( 4(a+(\chi_2\lambda_2v_2-\chi_1v_1\lambda_1)(x,t;u_0)-|\nabla(\chi_2v_2-\chi_1v_1)(x,t;u_0)|^2\right)\nonumber \\
&\geq 4\big(a-\frac{a\left[(\chi_2\mu_2\lambda_2-\chi_1\mu_1\lambda_1)_{-}+\chi_2\mu_2(\lambda_1-\lambda_2)_{-}\right]}{\lambda_1(b+\chi_2\mu_2-\chi_1\mu_1-M)}\big)-\frac{ND^{2}a^2}{(b+\chi_2\mu_2-\chi_1\mu_1-M)^2}.
\end{align}
The Lemma thus follows.
\end{proof}

\begin{proof}[Proof of Theorem D]
 The arguments used in this proof generalize some of the arguments used in the proof of Theorem 9(i) \cite{SaSh1}. Hence some details might be omitted. We refer the reader to \cite{SaSh1} for the proofs of the estimates stated below.

Since \eqref{asym-hyp2} holds,  we have
\begin{align*}
c^{*}_{-}(\chi_1,\mu_1,\lambda_1,\chi_2,\mu_2,\lambda_2):=2\sqrt{a(1-L)}-\frac{aD\sqrt{N}}{b+\chi_2\mu_2-\chi_1\mu_1-M}>0,
\end{align*}
where $M$, $D$ and $L$ are given by \eqref{gl-exist-thm-eq1} and \eqref{asym-hyp1} and \eqref{L-eq} respectively. We first note that, it follows from Lemma \ref{asym-lem5} and the proof of Lemma 5.4 \cite{SaSh1} that  for every $0\leq c< c^*_{-}(\chi_1,\mu_1,\lambda_1,\chi_2,\mu_2,\lambda_2)$ we have
 \begin{equation}\label{asym-eq26}
 \liminf_{t\to \infty}\inf_{|x|\leq ct}u(x,t;u_0)>0.
 \end{equation}
 It suffices to prove the following claim.

 \smallskip

\noindent {\bf Claim. }  For every $0\leq c<c^*_{-}(\chi_1,\mu_1,\lambda_1,\chi_2,\mu_2,\lambda_2)$, we have that
\begin{equation}\label{asym-eq27}
\lim_{t\to\infty}\sup_{|x|\leq ct}|u(x,t;u_{0})-\frac{a}{b} |=0.
\end{equation}

Suppose that the claim is not true. Then there is  $0\leq c<c^*_{-}(\chi_1,\mu_1,\lambda_1,\chi_2,\mu_2,\lambda_2)$, $\delta>0$, a sequence $\{x_n\}_{n\geq 1},$ a sequence of positive numbers $\{t_n\}_{n\geq 1}$ with $t_n\to\infty$ as $n\to\infty$ such that \begin{equation}\label{asym-eq28}
|x_n|\leq ct_{n},\quad\forall\ n\geq 1,
\end{equation}
and
\begin{equation}\label{asym-eq29}
|u(x_n,t_n;u_0)-\frac{a}{b}|\geq \delta, \qquad  \forall\ n\geq 1.
\end{equation}
For every $n\geq 1$, let us define
\begin{equation}\label{asym-eq30}
u_n(x,t)=u(x+x_n,t+t_n;u_0),\, \,\,\, v_{k,n}(x,t)=v_k(x+x_n,t_n;u_0)\,\,\, (k=1,2)
\end{equation}
for all  $x\in\R^N$ and $t\geq-t_n$.

We first show that there is a subsequence of $\{(u_n,v_{1n},v_{2n})\}$ which converges locally uniformly. To this end,
let $\{T(t)\}_{t\geq 0}$ denote the analytic semigroup generated by the closed linear operator $(\Delta-I)u$ on $C^{b}_{\rm unif}(\R^N)$. Then the variation of constant formula yield that
\begin{align}\label{asym-eq31}
u(x,t;u_0)=&T(t)u_0+\int_0^t T(t-s)\nabla\cdot((\chi_2u\nabla v_2-\chi_1 u\nabla v_1)(\cdot,s;u_0))ds\nonumber\\
&\qquad\qquad  +\int_0^t  T(t-s)(((a+1)u-bu^2)(\cdot,s;u_0))ds.
\end{align}
Let $0<\alpha<\frac{1}{2}$ be fixed and let $X^{\alpha}$ denotes the fractional powers associated to the semigroup $\{T(t)\}_{t\geq0}$. Thus, there is a constant $C_{\alpha}$(see \cite{Dan Henry}) depending only on $\alpha$ and the dimension $N$ such that
\begin{eqnarray}\label{asym-eq32}
\|u_{n}(\cdot,0)\|_{X^{\alpha}}& \leq C_{\alpha}t_n^{-\alpha}\|u_0\|_{\infty} +C_{\alpha}\int_0^{t_n}e^{-(t_n-s)}(t_n-s)^{-\frac{1}{2}-\alpha}\|(\chi_2u\nabla v_{2}-\chi_1u\nabla v_1)(\cdot,s;u_0)\|_{\infty}ds\nonumber\\
& + C_{\alpha}\int_{0}^{t_n}e^{-(t_n-s)}(t_n-s)^{-\alpha}\|( (a+1)u-bu^2)(\cdot,s;u_0)\|_{\infty}ds.
\end{eqnarray}
Using the facts that $\sup_{t\geq 0}\|u(\cdot,t)\|_{\infty}<\infty$,  $t_n\to \infty$ as $n\to\infty$, $\int_0^{\infty}e^{-\tau}\tau^{-\frac{1}{2}-\alpha}d\tau=\Gamma(\frac{1}{2}-\alpha)<\infty$ and $\int_{0}^\infty e^{-\tau}\tau^{-\alpha}d\tau=\Gamma(1-\alpha)<\infty$, it follows from \eqref{asym-eq32} that
\begin{equation}
\sup_{n\geq 1}\|u_{n}(\cdot,0)\|_{X^\alpha}<\infty.
\end{equation}
Similar arguments as those used in the proof of Theorem 1.1 \cite{SaSh1} yield that the functions $u_n : [-T\ , \ T]\to X^{\alpha} $ are equicontinuous for every $T>0$. Hence Arzela-Ascili's Theorem and Theorem 15 (page 80 of \cite{Friedman}) imply that there is a function $(\tilde{u},\tilde{v}_1,\tilde{v}_2)\in \left[C^{2,1}(\R^N\times\R)\right]^3$ and a subsequence $\{(u_{n'},v_{1,n'},v_{2,n'})\}_{n\geq 1}$ of $\{(u_{n},v_{1,n},v_{2,n})\}_{n\geq 1}$ such that $(u_{n'},v_{1,n'},v_{2,n'})\to (\tilde{u},\tilde{v}_{1},\tilde{v}_{2})$ in $C_{loc}^{1+\delta',\delta'}(\R^N\times\R) $ for some $\delta'>0$. Moreover $\mu_i\tilde{u}=(\lambda_i I-\Delta)\tilde{v}_{i}$  for every $i=1,2$. Note that
$$
\tilde{u}(x,t)=\lim_{n\to\infty}u(x+x_{n'},t+t_{n'};u_0),\quad \forall\ x\in\R^N,\ \ t\in\R.
$$
Hence
\begin{equation}\label{asym-eq33}
|\tilde{u}(0,0)-\frac{a}{b}|\geq \delta.
\end{equation} Choose $\tilde{c}\in(c\ ,\ c^*_{-}(\chi_1,\mu_1,\lambda_1,\chi_2,\mu_2,\lambda_2))$. For every $x\in\R^N, t\in\R$ and $t_{n'}\geq \frac{|x|+\tilde{c}|t|}{\tilde{c}-c}$, we have
$$
|x+x_{n'}|\leq |x|+ct_{n'}\leq \tilde{c}(t_{n'}+t).
$$
It follows from last inequality and \eqref{asym-eq26} that
$$
\tilde{u}(x,t)=\lim_{n\to\infty}u(x+x_{n'},t+t_{n'};u_0)\geq \liminf_{s\to\infty}\inf_{|y|\leq \tilde{c}s}u(y,s;u_0)>0,\quad \forall\ x\in\R^N, \ \ t\in\R.
$$
Hence $\inf_{(x,t)\in\R^{N+1}}\tilde{u}(x,t)>0$.

 Next, we claim that $\tilde{u}(x,t)=\frac{a}{b}$ for every $x\in\R^N,\ t\in\R$. Indeed, let $\underline{u}_0=\inf_{(x,t)\in\R^{N+1}}\tilde{u}(x,t)$ and $\overline{u}_0(x,t)=\sup_{(x,t)\in\R^{N+1}}\tilde{u}(x,t)$. For every $t_0\in\R$, let $\overline{U}(t,t_0)$ and $\underline{U}(t,t_0)$ be the solution of the ODEs
\begin{equation}\label{asym-eq34}
\begin{cases}
\overline{U}_t= \big(a-(b+\chi_2\mu_2-\chi_1\mu_1)\overline{U}\big)\overline{U}\nonumber\\
\qquad\,\,\, +\frac{1}{\lambda_2}\Big(\left[(\chi_2\mu_2\lambda_2-\chi_1\mu_1\lambda_1)_{+}+\chi_1\mu_1(\lambda_1-\lambda_2)_{+}\right]\overline{u}_0 \nonumber\\
\qquad\,\,\,  -\left[(\chi_2\mu_2\lambda_2-\chi_1\mu_1\lambda_1)_{-}+\chi_1\mu_1(\lambda_1-\lambda_2)_{-}\right]\underline{u}_0\Big)\overline{U},\quad t>t_0\\
\overline{U}(t_0,t_0)=\overline{u}_0
\end{cases}
\end{equation}
and
\begin{equation}\label{asym-eq35}
\begin{cases}
\underline{U}_t=\big(a-(b+\chi_2\mu_2-\chi_1\mu_1)\underline{U}\big)\underline{U}\nonumber\\
\qquad\,\,\, +\frac{1}{\lambda_2}\Big(\left[(\chi_2\mu_2\lambda_2-\chi_1\mu_1\lambda_1)_{+}+\chi_1\mu_1(\lambda_1-\lambda_2)_{+}\right]\underline{u}_0 \nonumber\\
 \qquad \,\,\, -\left[(\chi_2\mu_2\lambda_2-\chi_1\mu_1\lambda_1)_{-}+\chi_1\mu_1(\lambda_1-\lambda_2)_{-}\right]\overline{u}_0\Big)\underline{U},\quad t>t_0\\
\underline{U}(t_0,t_0)=\underline{u}_0
\end{cases}
\end{equation}
respectively.
It follows from the arguments used to establish \eqref{asym-eq3} and \eqref{asym-eq5} that
\begin{equation}\label{asym-eq35}
\underline{U}(t-t_0,0)=\underline{U}(t,t_0)\leq \tilde{u}(x,t),\quad \forall\ x\in\R^N,\ t\geq t_0
\end{equation}
and
\begin{equation} \label{asym-eq36}
\overline{U}(t-t_0,0)=\overline{U}(t,t_0)\geq \tilde{u}(x,t),\quad \forall\ x\in\R^N,\ \ t\geq t_0
\end{equation}
respectively. Note that for every $t\in\R$ fixed, we have that
\begin{align} \label{asym-eq37}
\lim_{t_0\to-\infty}\overline{U}(t,t_0)=\frac{1}{b+\chi_2\mu_2-\chi_1\mu_1}
\Big\{&a+\frac{1}{\lambda_2}\big(\left[(\chi_2\mu_2\lambda_2-\chi_1\mu_1\lambda_1)_{+}+\chi_1\mu_1(\lambda_1-\lambda_2)_{+}\right]\overline{u}_{0} \nonumber\\ &-\left[(\chi_2\mu_2\lambda_2-\chi_1\mu_1\lambda_1)_{-}+\chi_1\mu_1(\lambda_1-\chi_2)_{-}\right]\underline{u}_0\big)\Big\}
\end{align}
and
\begin{align} \label{asym-eq38}
\lim_{t_0\to-\infty}\underline{U}(t,t_0)=\frac{1}{b+\chi_2\mu_2-\chi_1\mu_1}
\Big\{& a+\frac{1}{\lambda_2}\big(\left[(\chi_2\mu_2\lambda_2-\chi_1\mu_1\lambda_1)_{+}+\chi_1\mu_1(\lambda_1-\lambda_2)_{+}\right]\underline{u}_{0} \nonumber\\ &-\left[(\chi_2\mu_2\lambda_2-\chi_1\mu_1\lambda_1)_{-}+\chi_1\mu_1(\lambda_1-\chi_2)_{-}\right]\overline{u}_0\big)\Big\}.
\end{align}
Combining \eqref{asym-eq35} and \eqref{asym-eq38}, we have that
\begin{align}\label{asym-eq39}\frac{1}{b+\chi_2\mu_2-\chi_1\mu_1}
\Big\{&a+\frac{1}{\lambda_2}\big(\left[(\chi_2\mu_2\lambda_2-\chi_1\mu_1\lambda_1)_{+}+\chi_1\mu_1(\lambda_1-\lambda_2)_{+}\right]\underline{u}_{0} \nonumber\\ &-\left[(\chi_2\mu_2\lambda_2-\chi_1\mu_1\lambda_1)_{-}+\chi_1\mu_1(\lambda_1-\chi_2)_{-}\right]\overline{u}_0\big)\Big\}  \leq \underline{u}_0.
\end{align}
Combining \eqref{asym-eq36} and \eqref{asym-eq37}, we have that
\begin{align}\label{asym-eq40}
\frac{1}{b+\chi_2\mu_2-\chi_1\mu_1}
\Big\{ &a+\frac{1}{\lambda_2}\big(\left[(\chi_2\mu_2\lambda_2-\chi_1\mu_1\lambda_1)_{+}+\chi_1\mu_1(\lambda_1-\lambda_2)_{+}\right]\overline{u}_{0} \nonumber\\ & -\left[(\chi_2\mu_2\lambda_2-\chi_1\mu_1\lambda_1)_{-}+\chi_1\mu_1(\lambda_1-\chi_2)_{-}\right]\underline{u}_0\big)\Big\} \geq \overline{u}_0.
\end{align}
Thus, it follows from inequalities \eqref{asym-eq39} and \eqref{asym-eq40} that
\begin{equation}\label{asym-eq41}
\left(b+\chi_2\mu_2-\chi_1\mu_1-\frac{1}{\lambda_2}\Big( |\chi_1\mu_1\lambda_1-\chi_2\mu_2\lambda_2|+\chi_1\mu_1|\lambda_1-\lambda_2| \Big) \right)(\overline{u}_0-\underline{u}_0)\leq 0.
\end{equation}
Similarly, for every $t_0\in\R$, by considering $\overline{V}(t,t_0)$ and $\underline{V}(t,t_0)$ to be the solutions of the ODEs
\begin{equation}\label{asym-eq42}
\begin{cases}
\overline{V}_t= \big(a-(b+\chi_2\mu_2-\chi_1\mu_1)\overline{V}\big)\overline{V}\nonumber\\
\qquad\,\,  +\frac{1}{\lambda_1}\Big(\left[(\chi_2\mu_2\lambda_2-\chi_1\mu_1\lambda_1)_{+}+\chi_2\mu_2(\lambda_1-\lambda_2)_{+}\right]\overline{u}_0 \nonumber\\ \qquad \,\, -\left[(\chi_2\mu_2\lambda_2-\chi_1\mu_1\lambda_1)_{-}+\chi_2\mu_2(\lambda_1-\lambda_2)_{-}\right]\underline{u}_0\Big)\overline{V},\quad t>t_0\\
\overline{V}(t_0,t_0)=\overline{u}_0
\end{cases}
\end{equation}
and
\begin{equation}\label{asym-eq43}
\begin{cases}
\underline{V}_t=\big(a-(b+\chi_2\mu_2-\chi_1\mu_1)\underline{V}\big)\underline{V}\nonumber\\
\qquad\,\, +\frac{1}{\lambda_1}\Big(\left[(\chi_2\mu_2\lambda_2-\chi_1\mu_1\lambda_1)_{+}+\chi_2\mu_2(\lambda_1-\lambda_2)_{+}\right]\underline{u}_0 \nonumber\\
 \qquad\,\, -\left[(\chi_2\mu_2\lambda_2-\chi_1\mu_1\lambda_1)_{-}+\chi_2\mu_2(\lambda_1-\lambda_2)_{-}\right]\overline{u}_0\Big)\underline{V},\quad t>t_0\\
\underline{V}(t_0,t_0)=\underline{u}_0
\end{cases}
\end{equation}
respectively. Using systems \eqref{asym-eq42} and \eqref{asym-eq43}, similar arguments used to establish \eqref{asym-eq41} yield that \begin{equation}\label{asym-eq44}
\left(b+\chi_2\mu_2-\chi_1\mu_1-\frac{1}{\lambda_1}\Big( |\chi_1\mu_1\lambda_1-\chi_2\mu_2\lambda_2|+\chi_2\mu_2|\lambda_1-\lambda_2| \Big) \right)(\overline{u}_0-\underline{u}_0)\leq 0.
\end{equation}
It follows from inequalities \eqref{asym-eq41} and \eqref{asym-eq44} that
\begin{equation}\label{asym-eq45}
(b+\chi_2\mu_2-\chi_1\mu_1-K)(\overline{u}_0-\underline{u}_0)\leq 0.
\end{equation}
Since \eqref{main-asym-eq} holds, it follows from the last inequality that $\overline{u}_{0}=\underline{u}_0$. Combining this with inequalities \eqref{asym-eq39} and \eqref{asym-eq40} we obtain that $\overline{u}_0=\underline{u}_0=\frac{a}{b}$.   Hence, we have that $\tilde{u}(x,t)=\frac{a}{b}$ for every $x\in\R^N$ and $t\in\R$. In particular, we have that $\tilde{u}(0,0)=\frac{a}{b}$, which contradicts \eqref{asym-eq33}.

 Hence the claim is true and  Theorem D is thus proved.
\end{proof}

\end{document}